\documentclass[12pt,american]{article}
\usepackage[T1]{fontenc}
\usepackage[latin9]{inputenc}
\usepackage{color}
\usepackage{babel}
\usepackage{units}
\usepackage{mathrsfs}
\usepackage{amsmath,cases,stackrel}
\usepackage{amssymb}
\usepackage{tikz-cd}
\usepackage[unicode=true,bookmarks=true,bookmarksnumbered=false,bookmarksopen=false,
 breaklinks=true,pdfborder={0 0 1},backref=false,colorlinks=true]
 {hyperref}
\hypersetup{pdftitle={Cameron--Martin Type Theorem for a Class of non-Gaussian Measures},
 pdfauthor={Mohamed Erraoui and Michael R{\"o}ckner and Jos{\'e} Lu{\'\i}s da Silva},
 pdfsubject={Cameron-Martin theorem for non-Gaussian measures},
 pdfkeywords={Generalized grey Brownian motion, fractional Brownian motion, Wiener integral w.r.t. fBm, Integration by parts}, pdfborderstyle=, linkcolor=black, citecolor=blue, urlcolor=blue, filecolor=blue}
\usepackage{amsthm}
\makeatletter
\theoremstyle{plain}
\newtheorem{thm}{\protect\theoremname}[section]
\theoremstyle{definition}
\newtheorem{defn}[thm]{\protect\definitionname}
\theoremstyle{remark}
\newtheorem{rem}[thm]{\protect\remarkname}
\theoremstyle{plain}
\newtheorem{lem}[thm]{\protect\lemmaname}
\theoremstyle{plain}
\newtheorem{cor}[thm]{\protect\corollaryname}
\newtheorem{prop}[thm]{\protect\propositionname}

\usepackage{xcolor}
\usepackage{babel}
\usepackage{mathrsfs,bbm}

\usepackage{pdfsync}

\makeatother

\providecommand{\corollaryname}{Corollary}
\providecommand{\definitionname}{Definition}
\providecommand{\lemmaname}{Lemma}
\providecommand{\remarkname}{Remark}
\providecommand{\theoremname}{Theorem}
\providecommand{\propositionname}{Proposition}

\begin{document}
\title{Cameron--Martin Type Theorem for a Class of non-Gaussian Measures}
\author{\and \textbf{Mohamed Erraoui}\\
 Department of Mathematics, Faculty of Science El Jadida,\\
 Choua{\"i}b Doukkali University, Morocco\\
 Email: erraoui@uca.ac.ma\\
 \textbf{Michael R{\"o}ckner},\\
 Fakult{\"a}t f{\"u}r Mathematik, Universit{\"a}t Bielefeld,\\
 D-33501 Bielefeld, Germany.\\
 Email: roeckner@math.uni-bielefeld.de \\ and\\
 Academy of Mathematics and System Sciences, CAS, Beijing, China. \and
 \textbf{Jos{\'e} Lu{\'i}s da Silva},\\
 CIMA, Faculty of Exact Sciences and Engineering,\\
 University of Madeira, Campus da Penteada,\\
 9020-105 Funchal, Portugal.\\
 Email: joses@staff.uma.pt}
 
\maketitle
\begin{abstract}
In this paper, we study the quasi-invariant property of a class of non-Gaussian measures. These measures are associated with the family of generalized grey Brownian motions. We identify the Cameron--Martin space and derive the explicit Radon-Nikodym density in terms of the Wiener integral with respect to the fractional Brownian motion. Moreover,  we show an integration by parts formula for the derivative operator in the directions of the Cameron--Martin space. As a consequence, we derive the closability of both the derivative and the corresponding gradient operators.
\end{abstract}

\tableofcontents{}

\section{Introduction}
The main goal of this paper (see Theorem~\ref{thm:CM-ggBm} below) is to prove the Cameron--Martin theorem for the class of non-Gaussian processes, called generalized grey Brownian motion (ggBm) denoted by $B_{\beta,\alpha}$, $0<\beta\le1$ and $0<\alpha\le2$. Specifically, we are looking for suitable (random) shifts $\xi$ such that the distributions of $B_{\beta,\alpha}$ and $B_{\beta,\alpha}+\xi$ are equivalent. In other words, we are studying the quasi-invariance property of the ggBm law.

W.~Schneider \cite{Schneider90a, Schneider90, MR1190506} was the first to introduce this type of process, which he called grey Brownian motion (gBm). This provided stochastic models for slow-anomalous diffusion. A.~Mura then extended this to ggBm, which can be used as non-Markovian stochastic models for either slow or fast-anomalous diffusion; see \cite{Mura2008}.
The remarkable characteristic of ggBm is that it can be expressed in terms of $B^{\alpha/2}$, a fractional Brownian motion (fBm) with Hurst parameter $\alpha/2$. The first representation of ggBm is given in \cite{Mura_Pagnini_08} as the product of fBm $B^{\alpha/2}$ and a non-negative random variable $Y_\beta$ with $M$-Wright density function (see  Subsection~\ref{subsec:ggBm} below for details), that is,
\begin{equation}\label{eq:repre-ggBm-intro1}
    \big\{B_{\beta,\alpha}(t)\;t\ge0\big\}  \overset{\mathcal{L}}{=} \big\{\sqrt{Y_\beta}B^{\alpha/2}(t),\;t\ge0\big\},
\end{equation}
where $\overset{\mathcal{L}}{=}$ means equality in law. 
The representation \eqref{eq:repre-ggBm-intro1} is particularly interesting since it allows us to infer a variety of properties of ggBm $B_{\beta,\alpha}$ from those of fBm $B^{\alpha/2}$. For example, the H\"older continuity of the trajectories. 

An alternative representation of ggBm is given in terms of a subordination of fBm. The stochastic representation through subordinated processes is very natural, as it gives a direct physical interpretation, see Remark 9.3 in \cite{Mura2008}. However, the subordinated representations already given for ggBm in \cite{Mura_Taqqu_Mainardi_08, DaSilva2014,daSilva2020} represent ggBm only in one dimensional time marginal law. Therefore, they cannot characterize the complete stochastic structure of the process, nor can they substitute ggBm when more than one marginal law is involved, as in the Cameron-Martin theorem.
 So, another representation involving finite-dimensional distributions is needed for the problem that concerns us here. Our investigation yields the following representation (see Proposition~\ref{prop:subord-ggBm})
\begin{equation}\label{eq:repre-ggBm-intro2}
    \big\{B_{\beta,\alpha}(t)\;t\ge0\big\}  \overset{\mathcal{L}}{=} \big\{B^{\alpha/2}(tY_\beta^{1/\alpha}),\;t\ge0\big\},
\end{equation}
 which allows us to achieve the main goal of the paper, that is, the Cameron--Martin theorem for ggBm, the integration by parts formula, and the closability of the derivative and gradient operators, see Sections~\ref{sec:Cameron-Martin-formula} and \ref{sec:IbP-closability} below.   
Let us now explain the approach we adopt. 

First, the general setting is as follows:
Let $\mathbb{W}$ be the classical Wiener space, $\mathcal{B}(\mathbb{W})$ its Borel $\sigma$-algebra, and $H=\alpha/2\in [1/2,1)$. Furthermore, let $\mathbb{P}_H$ be the unique probability measure on $\mathbb{W}
$ such that the canonical process $W^H$ is a fBm. The distribution of the random variable $Y_\beta$ is denoted by $\mathbb{P}_{Y_\beta}$ and $Y_\beta$ can be realized on $\mathbb{R}^+$ as the identity map
\begin{equation}\label{eq:canonical-space}
\mathcal{Y}_{\beta}:\mathbb{R}^{+}\longrightarrow\mathbb{R}^{+},\;\tau\mapsto\mathcal{Y}_\beta(\tau):=\tau.
\end{equation}
The ggBm is realized on the probability space $\big(\mathbb{W}\times\mathbb{R}^+,\mathcal{B}(\mathbb{W})\otimes\mathcal{B}(\mathbb{R}^+),\mathbb{P}_H\otimes\mathbb{P}_{Y_\beta}\big)$ as the canonical process $X_{\beta,2H}$ defined by
\[
X_{\beta,2H}(t)(w,\tau)=W^{H}\big(t\mathcal{Y}_{\beta}^{1/(2H)}(\tau)\big)(w)=w\big(t\tau^{1/(2H)}\big),\; t,\tau\in\mathbb{R}^{+},\; w\in\mathbb{W}.
\]
After considering three essential elements - the representation \eqref{eq:repre-ggBm-intro2} for ggBm, the Cameron--Martin type theorems for fBm, and subordinated Bm (\cite{Deng2015}) - we conclude that the natural choice for $\xi$ is $h(t\mathcal{Y}_\beta)$, $t\ge0$, where $h\in\mathcal{H}_{\alpha/2}$. Here $\mathcal{H}_{\alpha/2}$ is the Cameron--Martin space associated with fBm $B^{\alpha/2}$; see Subsection~\ref{subsec:Wiener-int-fBm}. Our proof is based on an appropriate use of the exponential martingale and the Cameron--Martin theorem for fBm.

It is a well-known fact that an important consequence of the quasi-invariant property is the integration by parts formula.
This is what we have obtained in Theorems~\ref{thm:IbP-gBm} and \ref{thm:IbP-ggBm} for the partial derivative of a bounded smooth cylinder functions $F:\mathbb{W}\times\mathbb{R}^+\longrightarrow\mathbb{R}$ in the direction of $h\in\mathcal{H}_H$ defined by
\[
(\partial_{1,h}F)(w,\tau):=\lim_{\varepsilon\to0}\frac{F(w+\varepsilon h,\tau)-F(w,\tau)}{\varepsilon},\quad w\in\mathbb{W}, \tau\in\mathbb{R}^+.
\]
Finally, using the integration by parts formula, we prove in Theorem~\ref{thm:closable-gradient-ggBm} the closability on $L^p(\mathbb{W}\times\mathbb{R}^+,\sigma\big(X_{\beta,2H}(t),t\in\left[0,T\right]\big),\mathbb{P}_{H}\otimes\mathbb{P}_{Y_\beta})$, $p\geq 1$, of the gradient $\nabla_1$, the unique element in $ \mathcal{H}_H$  verifying
\[
(\partial_{1,h}F)(w,\tau)=\big(\nabla_1 F(w,\tau),h\big)_{\mathcal{H}_{H}},\quad w\in\mathbb{W},\;\tau\in\mathbb{R}^+,
\]
where $(\cdot,\cdot)_{\mathcal{H}_H}$ denotes the inner product given in Eq.~\eqref{eq:scalarproductH0}; see also the diagram in Fig.~\ref{diagram}.

The paper is organized as follows. In Section~\ref{sec:Preliminaries} we recall the definition and key properties of fBm that will be needed later. The Cameron--Martin space of the fBm is recalled and characterized in some detail. The Wiener integral with respect to fBm and the Cameron--Martin theorem of fBm are presented. The class of ggBm is described as well as their canonical realization. In Section~\ref{sec:Cameron-Martin-formula} we show the main results of the paper concerning the Cameron--Martin theorem for the class of ggBm. Section~\ref{sec:IbP-closability} contains an integration by parts formula of the derivative in the directions of the Cameron-Martin space and the closability of both the derivative and the corresponding gradient.
\\[.2cm]
{\bf Notation.} In what follows, we denote by $(\Omega,\mathcal{F},\mathbb{P})$ a
complete probability space. In addition, we consider the Hilbert space
$L^{2}(\lambda):=L^{2}(\mathbb{R}^{+},\lambda;\mathbb{R})$, of real-valued square integrable Borel functions on $\mathbb{R}^{+}$ with respect to (w.r.t.) the Lebesgue
measure $\lambda$. The scalar product in $L^2(\lambda)$ is denoted by $(\cdot,\cdot)_{2}$
and the associated norm by $\|\cdot\|_{2}$. By $(\mathbb{W},\mathcal{B}(\mathbb{W}))$
we denote the classical Wiener space. More precisely, the path space 
\[
\mathrm{\mathbb{W}}=\big\{ w:\mathbb{R}^{+}\longrightarrow\mathbb{R}\mid w\;\mathrm{is\;continuous\;and}\;w(0)=0\big\},
\]
endowed with the locally uniform convergence topology and $\mathcal{B}(\mathbb{W})$
denotes the associated Borel $\sigma$-algebra. 
Furthermore, for every $T>0$, we define 
\[
\mathbb{W}_{T}:=\left\{ w\!\!\upharpoonright_{\left[0,T\right]}\, \mid w\in\mathbb{W}\right\},
\]
where $w\!\!\upharpoonright_{\left[0,T\right]}$ is the restriction of $w$ to the interval $[0,T]$.
For a continuous process
$X,$ by $\mathbb{P}_{X}$ we mean the law of the process $X$ on
$\mathbb{W}$. 
\section{Preliminaries}
\label{sec:Preliminaries}

In this section, we discuss the types of processes we are working with, namely the fBm and ggBm. In addition, the Cameron--Martin space and the Wiener integral w.r.t.~fBm are reviewed. Finally, two useful representations of ggBm and its canonical realization are shown.

\subsection{Fractional Brownian Motion}

In this subsection, we review the definition and important properties of the fBm motion that are necessary for what follows. For more
details, see, for example, the book \cite{mishura08}, and references
therein.
\begin{defn}[Fractional Brownian motion]
\label{def:fBm}A (one-side, normalized) fBm with Hurst index $H\in\big[1/2,1)$ is a Gaussian process $B^{H}=\big\{ B^{H}(t),t\in\mathbb{R}^{+}\big\}$
on $(\Omega,\mathcal{F},\mathbb{P})$, satisfying the properties 
\begin{enumerate}
\item $B^{H}(0)=0$, $\mathbb{P}$-a.s., that is, $B^{H}$ starts at zero
almost surely. 
\item $\mathbb{E}[B^{H}(t)]=0,$ $t\in\mathbb{R}^{+}$, the process $B^{H}$
is centered.
\item The covariance kernel $R_{H}$ of $B^{H}$ is given, for any $s,t\in\mathbb{R}^{+}$, by
\begin{equation}
R_{H}(t,s):=\mathbb{E}\big[B^{H}(t)B^{H}(s)\big]=\frac{1}{2}\big(t^{2H}+s^{2H}-|t-s|^{2H}\big).\label{eq:covariance-kernel-fBm}
\end{equation}
\end{enumerate}
\end{defn}

\begin{rem}
\label{rem:fBm}
\begin{enumerate}
\item Note that $\mathbb{E}\big[\big(B^{H}(t)-B^{H}(s)\big)^{2}\big]=|t-s|^{2H}$
and $B^{H}$ is a Gaussian process, then $B^{H}$ has a modification with continuous trajectories, according to the Kolmogorov theorem. We will always consider such a modification and keep the same notation.
\item In general, for every $n\geq1$, it holds 
\[
\mathbb{E}\big[|B^{H}(t)-B^{H}(s)|^{n}\big]=\sqrt{2^{n}\pi^{-1}}\Gamma\left(\frac{n+1}{2}\right)|t-s|^{nH}.
\]
\item For $H=1/2$ the process $B^{1/2}$ becomes the standard
Brownian motion (Bm) with covariance kernel $R_{1/2}(t,s)=t\wedge s$.
\item For any $a>0$, the process $\{a^{-H}B^{H}(at),t\in\mathbb{R}^{+}\}$ has the same
distribution as $B^{H}$. In other words, $B^{H}$ is a $H$-self-similar
process.
\item The characteristic function has the form 
\[
C_{\lambda}(t):=\mathbb{E}\left[\exp\left(\mathrm{i}\sum_{k=1}^{n}\lambda_{k}B^{H}(t_{k})\right)\right]=\exp\left(-\frac{1}{2}\big(\lambda,\Sigma_{H}(t)\lambda^{\top}\big)\right),
\] 
where $\lambda=(\lambda_{1},\dots,\lambda_{n})\in\mathbb{R}^{n}$,
$\lambda^{\top}$ is the transpose of $\lambda$, $t_{k}\ge0$, $k=1,\dots,n$,
$\Sigma_{H}(t)=(R_{H}(t_{k},t_{j}))_{k,j=1}^{n}$ is the covariance
matrix and $(\cdot,\cdot)$ is the scalar product on $\mathbb{R}^{n}$.
\item We denote by $\mathbb{P}_H:=\mathbb{P}_{B^H}$ the unique probability measure in $(\mathbb{W},\mathcal{B}(\mathbb{W}))$ such that the canonical process
\[
W(t)(w):=w(t),\quad\forall t\ge0,\;w\in\mathbb{W},
\]
is a fractional Brownian motion
which we denote by $W^{H}$. For $H=1/2$, $\mathbb{P}_{1/2}$ becomes the Wiener measure. 
\end{enumerate}
\end{rem}
The integral representation of the covariance kernel $R_{H}$ plays
an important role in the construction of the Cameron--Martin space introduced in the following. We state this representation of $R_{H}$ in the following lemma.
\begin{lem}[cf.~\cite{DU99} Lemma 3.1]
The covariance kernel $R_{H}$ (also called reproducing kernel Hilbert
space (RKHS)) has the following integral representation
\begin{equation}
R_{H}(t,s)=\int_{0}^{t\wedge s}K_{H}(t,r)K_{H}(s,r)\,\mathrm{d}r,\quad t,s\ge0.\label{eq:cov-integral-kernel}
\end{equation}
The corresponding kernel $K_{H}$ is given by 
\begin{enumerate}
\item  For $H=1/2$
\begin{equation}\label{eq:kernel-K_1/2}
K_{1/2}(t,r)=\mathbbm{1}_{[0,t]}(r).
\end{equation}
\item  For $H\in (1/2,1)$ (see Lemma 4.3 in \cite{Decreusefond22})
\begin{equation}\label{eq:kernel-K_H}
K_{H}(t,r)=\frac{r^{1/2-H}}{\Gamma(H-1/2)}\int_{r}^{t}u^{H-1/2}(u-r)^{H-3/2}\mathrm{~d}u\,\mathbbm{1}_{(0,t]}(r).
\end{equation}
\end{enumerate}
\end{lem}
\subsection{Wiener Integral with Respect to Fractional Brownian Motion}
\label{subsec:Wiener-int-fBm}
Now we introduce the Cameron--Martin space of the fBm with $H\ge 1/2$. It plays an important role in defining the Wiener integral w.r.t.~fBm and the Cameron--Martin theorem below; more details can be found in \cite{Barton1988}, \cite{Decreusefond22} and the references therein.

\begin{defn}[Cameron-Martin space of the fBm]
Let $\mathfrak{R}_H$ be the vector space spanned by the covariance kernel $R_H$, that is,
\[
\mathfrak{R}_H:=\mathrm{span}\{R_H(t,\cdot)\mid t\ge0\}
\]
equipped with the scalar product
\begin{equation}\label{eq:scalarproductH0}
\big(R_H(t\,\cdot),R_H(s,\cdot)\big)_{\mathfrak{R}_H}:=R_H(t,s).   
\end{equation}
The Cameron--Martin space of the fBm with Hurst index $H$, denoted by $\mathcal{H}_H$, is the completion of $\mathfrak{R}_H$ with respect to the norm associated with the scalar product \eqref{eq:scalarproductH0}.

\end{defn}

\begin{rem}
The Cameron--Martin space, as defined above, is not very practical. It is possible to have a more suitable characterization of $\mathcal{H}_H$, see Theorem~3.3 in \cite{DU99} (in the special case of the time interval $[0,1]$) or Theorem~3.1 in \cite{Barton1988} for a general time interval. 
\end{rem}
Therefore, our goal is to obtain a more convenient description of the Cameron--Martin space $\mathcal{H}_H$.
To this end, we first introduce the linear operator $\mathcal{K}_H$ 
associated with the kernel $K_H$. More precisely, for every $f\in L^2(\lambda)$, we define
\[
(\mathcal{K}_Hf)(t) :=\int_0^tK_H(t,s)f(s)\,\mathrm{d}s.
\]
Note that using \eqref{eq:cov-integral-kernel}, the covariance kernel $R_H$ is the image of $K_H(t,\cdot )$ under $\mathcal{K}_{H}$, that is,
\begin{equation}\label{eq:CalK-K-equal-R}
\mathcal{K}_{H}(K_{H}(t,\cdot))(s)=R_H(t,s),\quad t,s\ge0.    
\end{equation}
Hence $K_H(t,\cdot)\in L^2(\lambda)$ for any $t\ge0$, indeed
\[
\|K_H(t,\cdot)\|_2^2=R_H(t,t)=t^{2H},
\]
ensuring that the operator $\mathcal{K}_{H}$ is well-defined on $L^{2}(\lambda)$ for every $t\ge0$.
This implies that 
\begin{equation}\label{eq:KK-R}
    \mathcal{K}_H(\mathfrak{K}_H)=\mathfrak{R}_H,
\end{equation}
where $\mathfrak{K}_H$ is the vector space spanned by the kernel $K_H(t,\cdot)$, that is,
\[
\mathfrak{K}_H:=\mathrm{span}\{K_H(t,\cdot)\mid t\ge0\}.
\]
Using the explicit forms of $K_H$ in \eqref{eq:kernel-K_1/2}, \eqref{eq:kernel-K_H} and applying the Fubini theorem (when $H>1/2$), $\mathcal{K}_Hf$ can be expressed as follows, with $t\ge0$
\begin{subnumcases}{(\mathcal{K}_Hf)(t)=} 
\displaystyle \int_0^tf(s)\,\mathrm{d}s, & $\displaystyle H=\frac{1}{2}$, \label{eq:Kop-Hequal12}\\[.2cm]
    \displaystyle\int_0^t \frac{x^{H-1/2}}{\Gamma(H-1/2)}\int_0^x (x-r)^{H-3/2} r^{1/2-H}f(r)\,\mathrm{d}r\,\mathrm{d}x, & $\displaystyle H>\frac{1}{2}$. \label{eq:Kop-Hplus12}
   \end{subnumcases}
See, for example, \cite{DU99} for \eqref{eq:Kop-Hplus12}.
It is not hard to see from the equalities \eqref{eq:Kop-Hequal12} and \eqref{eq:Kop-Hplus12} that 
\begin{equation}\label{eq:KH-one-to-one}
    \forall t\ge0,\; \int_0^tK_H(t,s)f(s)\,\mathrm{d}s=0\;\Longrightarrow\;f=0,\; \mathrm{a.e.}
\end{equation}
meaning that the operator $\mathcal{K}_H$ is one-to-one and $\mathfrak{K}_H$ is dense in $L^2(\lambda)$ for $H\ge 1/2$. As a consequence, $\mathcal{K}_H:L^2(\lambda)\longrightarrow\mathcal{K}_H\big(L^2(\lambda)\big)$ is a bijective isometry when $\mathcal{K}_H\big(L^2(\lambda)\big)$ is provided with the scalar product
\[
(\mathcal{K}_H h,\mathcal{K}_H g)_{\mathcal{K}_H(L^2(\lambda))}:=(h,g)_2,\quad h, g\in L^2(\lambda).
\]
On the other hand, \eqref{eq:KK-R} and \eqref{eq:KH-one-to-one} imply that  
\[\mathcal{K}_H:\big(\mathfrak{K}_H,(\cdot,\cdot)_2\big)\longrightarrow \big(\mathfrak{R}_H,(\cdot,\cdot)_{\mathfrak{R}_H}\big)
\]
is a bijective isometry. By a density argument, $\mathcal{K}_H$ is extended to an isomorphism between the Hilbert spaces $L^2(\lambda)$ and $\mathcal{H}_H$, that is,
\[
\tilde{\mathcal{K}}_H:\big(L^2(\lambda),(\cdot,\cdot)_2\big)\longrightarrow \big(\mathcal{H}_H,(\cdot,\cdot)_{\mathcal{H}_H}\big).
\]
The commutative diagram in Fig.~\ref{diagram} summarizes the above considerations and enhances the understanding of the identification of the Cameron--Martin space.

\begin{figure}
\begin{center}
\begin{tikzcd}[sep=huge]
\big(\mathfrak{K}_{H},\|\cdot\|_{2}\big)\ar[r, hook,two heads,"i",line width=.8pt]\ar[d,"\mathcal{K}_H"',line width=.8pt,shift right=.8em]& 
\big(L^2(\lambda),\|\cdot\|_2\big)
\ar[r,"\mathcal{K}_H",line width=.8pt] 
\dar[shift left=-.8em,line width=.8pt,"\mathcal{\tilde{K}}_H"']&
\big(\mathcal{K}_H \big(L^2(\lambda)\big),\|\cdot\|_{\mathcal{K}_H( L^2(\lambda))}\big)
\ar[ld,red,shift right=-.7em,line width=.8pt]\lar[shift left=.7em,line width=.8pt]\\
\big(\mathfrak{R}_{H},\|\cdot\|_{\mathfrak{R}_{H}}\big)
\rar[two heads,line width=.8pt,"\mathrm{Completion\;\;}"']
\uar[line width=.8pt] & 
\big(\mathcal{H}_H,\|\cdot\|_{\mathcal{H}_H}\big)\quad
\uar[line width=.8pt]
\urar[red,"\mathcal{K}_H\circ\mathcal{\tilde{K}}_H^{-1}" black,line width=.8pt]
\end{tikzcd}    
\end{center}
\caption{\label{diagram}Diagram with the identification of the Cameron--Martin space. An arrow with a hook means that the map is one-to-one. A double head indicates that the map is onto or its range is dense. Two arrows pointing in opposite directions means that the map is an isometric isomorphism. $\mathcal{\tilde{K}}_H$ is the extension of $\mathcal{K}_H$ to the complete space $L^2(\lambda)$.}
\end{figure}
We are ready to state the following useful characterization of the Cameron--Martin space.

\begin{thm}[Characterization of $\mathcal{H}_H$]
\label{thm:Cameron-martin-space-fBm}
 The Cameron--Martin space $\mathcal{H}_H$ can be identified with $\mathcal{K}_H(L^2(\lambda))$, the space of functions $h=\mathcal{K}_H\dot{h}$, $\dot{h}\in L^2(\lambda)$, given by
\begin{equation}\label{eq:h-rep-fBm}
h(t)=(\mathcal{K}_H\dot{h})(t)=\int_0^tK_H(t,s)\dot{h}(s)\,\mathrm{d}s,\quad t\ge0,    
\end{equation}
equipped with the inner product $(\cdot,\cdot)_{\mathcal{K}_H(L^2(\lambda))}$.
\end{thm}

\begin{rem}
\label{rem:Cameron-Martin-space-fBm}
\begin{enumerate}

\item With the identification given in \eqref{eq:h-rep-fBm}, for every $h\in \mathcal{H}_H$ the function $\dot{h}\in L^2(\lambda)$ is given by $\dot{h}:=\mathcal{K}_H^{-1}h$.  

\item In the special case $H=1/2$, $\mathcal{H}_{1/2}$
is the well-known space $AC(\mathbb{R}^+,\mathbb{R})$ of absolutely continuous functions, vanishing at $0$, whose derivative belongs to $L^{2}(\lambda)$. The representation \eqref{eq:h-rep-fBm} takes the form 
\begin{equation}\label{eq:1/2-rep}
   h(t)=\int_{0}^{t}\,\dot{h}(s)\,\mathrm{d}s,\quad   t\geq 0, 
\end{equation}
where $\dot{h}$ is the derivative in the sense of distributions of $h$.
\end{enumerate}
\end{rem}

We introduce the Wiener integral w.r.t.~$B^H$, $H>1/2$, following the approach in \cite{DU99, Decreusefond22, Barton1988} for which we address the interested reader for more details.

\begin{defn}[Wiener integral]
\label{def:Wiener-integral-fBm}The Wiener integral w.r.t.~fBm is defined  
as the extension to $\mathcal{H}_H$ of the isometry 
\[
\delta_{H}:\mathcal{K}_{H}(\mathfrak{K}_{H})\longrightarrow L^{2}(\mathbb{P})
\]
 defined by 
\[
R_H(t,\cdot)=\mathcal{K}_{H}(K_{H}(t,\cdot))\mapsto\delta_{H}\big(\mathcal{K}_{H}(K_{H}(t,\cdot))\big):=B^{H}(t),\quad \forall t\ge0.
\]
By linearity we have
\[
\delta_H\left(\sum_{i=1}^n a_i\mathcal{K}_H(K_H(t_i,\cdot))\right)= \sum_{i=1}^n a_i B^{H}(t_i),\quad a_i\in\mathbb{R},\; i=1,\dots,n.
\]
In general, for every (deterministic) function $h\in \mathcal{H}_H$, there exists a sequence $h_n:=\mathcal{K}_H(\dot{h}_n)$, $\dot{h}_n\in \mathfrak{K}_H$, $n\in \mathbb{N}$, such that $h=\mathcal{H}_H\mbox{-}\lim_{n\to\infty}h_n$ which leads to the Wiener integral w.r.t.~fBm of $h$ 
\[
\int_0^\infty h(t)\,\mathrm{d}B^H(t):=\delta_H(h):=L^2(\mathbb{P})\mbox{-}\lim_{n\to\infty}\delta_H(h_n)=L^2(\mathbb{P})\mbox{-}\lim_{n\to\infty}\delta_H(\mathcal{K}_H(\dot{h}_n)).
\]
Therefore, $\delta_H(h)$ is a centered Gaussian random variable with variance equal to $\|h\|_{\mathcal{H}_H}^{2}$.
\end{defn}
\begin{rem}\label{rem:Wiener-integral} It is easy to conclude from the definition of Wiener integral that:
\begin{enumerate}
    \item The process 
\begin{equation}
B:=\{\delta_{H}(\mathcal{K}_{H}(\mathbbm{1}_{[0,t]}))\mid t\ge0\}\label{eq:Bm-associated-fBm}
\end{equation}
 is a standard Brownian motion and 
\[
\delta_{H}(\mathcal{K}_{H}(u))=\int_{0}^{\infty}u(s)\,\mathrm{d}B(s),\quad u\in L^2(\lambda),
\]
where the integral on the right side is taken in the Wiener--It{\^o}
sense. In particular, we have
\begin{equation}
B^{H}(t)=\int_{0}^{t}K_{H}(t,s)\,\mathrm{d}B(s).\label{eq:repr-fBm-Bm}
\end{equation}
    \item Using the Remark~\ref{rem:Cameron-Martin-space-fBm}-1 we have, for every
$h\in\mathcal{H}_{H}$
\begin{equation}\label{eq:Wiener-integral-fBm-Bm}
\int_0^\infty h(t)\,\mathrm{d}B^H(t)=\int_{0}^{\infty}(\mathcal{K}_H^{-1}h)(t)\,\mathrm{d}B(t)=\int_{0}^{\infty}\dot{h}(t)\,\mathrm{d}B(t).    
\end{equation}
For more details, see \cite[Lemma~4.4]{Decreusefond22}.
\end{enumerate}
\end{rem}

Now we are ready to state the well-known Cameron--Martin theorem for fBm, cf.~\cite[Theorem~4.10]{Decreusefond22} or \cite[Theorem 4.1]{DU99}.

\begin{thm}
\label{thm:Cameron-Martin-thm}For every $h\in\mathcal{H}_{H}$, $H>1/2$, and
any bounded measurable functional $F:\mathbb{W}\longrightarrow\mathbb{R}$ we have
\begin{eqnarray}\label{eq:Cameron-Martin-fBm}
\mathbb{E}[F(B^{H}+h)] &=&\mathbb{E}\left[F(B^{H})\exp\left(\int_0^\infty h(t)\,\mathrm{d}B^H(t)-\frac{1}{2}\|h\|_{\mathcal{H}_{H}}^{2}\right)\right] \nonumber\\
&\overset{\eqref{eq:Wiener-integral-fBm-Bm}}{=}&\mathbb{E}\left[F(B^{H})\exp\left(\int_{0}^{\infty}\dot{h}(s)\,\mathrm{d}B(s)-\frac{1}{2}\|\dot{h}\|_{2}^{2}\right)\right].
\end{eqnarray}
\end{thm}
\begin{rem}\label{rem:exp-martinal-infty}
Since for every $h=\mathcal{K}_H(\dot{h})\in \mathcal{H}_H$ with $\dot{h}\in L^2(\lambda)$, the exponential martingale $\left\{\mathcal{E}\left(\displaystyle{\int_{0}^{\cdot}}\dot{h}(s)\,\mathrm{d}B(s)\right)(t),t\geq 0\right\}$ is uniformly integrable, therefore 
\[
\mathcal{E}\left(\displaystyle{\int_{0}^{\cdot}}\dot{h}(s)\,\mathrm{d}B(s)\right)(\infty)=\exp\left(\int_{0}^{\infty}\dot{h}(t)\,\mathrm{d}B(t)-\frac{1}{2}\int_{0}^{\infty}|\dot{h}(t)|^{2}\,\mathrm{d}t\right),
\] 
is well-defined. Thus, formula \eqref{eq:Cameron-Martin-fBm} can be rewritten as follows
\begin{equation}\label{eq:CMF-fBm}
\mathbb{E}[F(B^{H}+h)] =\mathbb{E}\left[F(B^{H})\mathcal{E}\left(\displaystyle{\int_{0}^{\cdot}}\dot{h}(s)\,\mathrm{d}B(s)\right)(\infty)\right].
\end{equation}
\end{rem}
\begin{rem}\label{rem:RKHS-finite-time}
    The above construction may also be realized if we consider a finite time interval $[0,T]$, $T>0$. That is,
    \begin{enumerate}
        \item the Cameron--Martin space $\mathcal{H}_{H}(T)$ is identified as the Hilbert space $\mathcal{K}_{H}\big(L^{2}([0,T],\lambda)\big)$ with the scalar product: 
        \[
        (h,g)_{\mathcal{H}_H(T)}=(\mathcal{K}_H \dot{h},\mathcal{K}_H \dot{g})_{\mathcal{K}_H(L^2([0,T],\lambda))}:=(\dot{h},\dot{g})_{L^2([0,T],\lambda)},
        \]
        where $\dot{h}, \dot{g}\in L^2([0,T],\lambda)$,
        \item the Wiener integral of $h\in \mathcal{H}_H(T)$: 
        \[
        \displaystyle{\int_0^T} h(s)\,\mathrm{d}B^H(s):=\displaystyle{\int_0^T} \dot{h}(s)\,\mathrm{d}B(s),
        \]
        \item the Cameron--Martin theorem says: for every $h\in \mathcal{H}_H(T)$ and any bounded functional $F:\mathbb{W}_T\longrightarrow\mathbb{R}$ we have
        \begin{eqnarray}
        \mathbb{E}[F(B^{H}+h)]&=&\mathbb{E}\left[F(B^{H})\exp\left(\int_{0}^{T}{h}(s)\,\mathrm{d}B^H(s)-\frac{1}{2}\|h\|^2_{\mathcal{H}_H(T)}\right)\right] \nonumber\\
        &=&\mathbb{E}\left[F(B^{H}) \mathcal{E}\left(\int_{0}^{.}\dot{h}(s)\,\mathrm{d}B(s)\right)\!(T)\right],\label{eq:CMF-fBm-0T}
        \end{eqnarray}
where $\mathcal{E}\left(\displaystyle{\int_{0}^{.}}\dot{h}(s)\,\mathrm{d}B(s)\right)$
is the exponential martingale associated with the martingale $\left\{ \displaystyle{\int_{0}^{t}}\dot{h}(s)\,\mathrm{d}B(s),t\in [0,T]\right\} $, that is, 
\[
\mathcal{E}\left(\int_{0}^{.}\dot{h}(s)\,\mathrm{d}B(s)\right)(t)=\exp\left(\int_{0}^{t}\dot{h}(s)\,\mathrm{d}B(s)-\frac{1}{2}\int_{0}^{t}|\dot{h}(s)|^{2}\,\mathrm{d}s\right).
\]
\end{enumerate}
   For more details, see, for example, \cite{Decreusefond22} and \cite{Coutin2007}.
\end{rem}
\begin{rem}
For $H=\frac{1}{2}$, it follows from \eqref{eq:repr-fBm-Bm} that
$B^{\frac{1}{2}}$ and $B$ are the same process. Therefore, to be consistent with the Wiener-It{\^o} integral we identify $\mathcal{H}_{1/2}$  with $L^2(\lambda)$, so in Equation~\eqref{eq:Wiener-integral-fBm-Bm} $\dot{h}$ should be taken as $h$. Hence, Equation~\eqref{eq:Cameron-Martin-fBm} includes the classical Cameron--Martin formula for the Brownian motion, as well as its equivalent forms~\eqref{eq:CMF-fBm-0T} and \eqref{eq:CMF-fBm}.
\end{rem}

\subsection{Generalized Grey Brownian Motion} \label{subsec:ggBm}
\subsubsection{Definition and Representations}
Let $0<\beta<1$ and $1\leq \alpha<2$ be given. 
A continuous stochastic
process defined on $\left(\Omega,\mathcal{F},\mathbb{P}\right)$
is called a ggBm, denoted by $B_{\beta,\alpha}=\{B_{\beta,\alpha}(t),\,t\geq0\}$,
see \cite{Mura_mainardi_09}, if: 
\begin{enumerate}
\item $B_{\beta,\alpha}(0)=0$, $\mathbb{P}$-a.s. 
\item Any collection $\big\{ B_{\beta,\alpha}(t_{1}),\ldots,B_{\beta,\alpha}(t_{n})\big\}$
with $0\leq t_{1}<t_{2}<\ldots<t_{n}<\infty$ has a characteristic function
given, for any $\theta=(\theta_{1},\ldots,\theta_{n})\in\mathbb{R}^{n}$,
by 
\begin{equation}
\mathbb{E}\left(\exp\left(i\sum_{k=1}^{n}\theta_{k}B_{\beta,\alpha}(t_{k})\right)\right)=E_{\beta}\left(-\frac{1}{2}\theta^{\top}\Sigma_{\alpha,n}\theta\right),\label{eq:charact-func-ggBm}
\end{equation}
where 
\[
\Sigma_{\alpha,n}=\big(t_{k}^{\alpha}+t_{j}^{\alpha}-|t_{k}-t_{j}|^{\alpha}\big)_{k,j=1}^{n}
\]
and $E_{\beta}$ is the Mittag-Leffler (entire) function 
\[
E_{\beta}(z)=\sum_{n=0}^{\infty}\frac{z^{n}}{\Gamma(\beta n+1)},\quad z\in\mathbb{C}.
\]
\end{enumerate}

The generalized grey Brownian motion has the following properties: 
\begin{enumerate}
\item For each $t\geq0$, the moments of any order are given by 
\[
\begin{cases}
\mathbb{E}[B_{\beta,\alpha}^{2n+1}(t)] & =0,\\
\noalign{\vskip4pt}\mathbb{E}[B_{\beta,\alpha}^{2n}(t)] & =\frac{(2n)!}{2^{n}\Gamma(\beta n+1)}t^{n\alpha}.
\end{cases}
\]
\item The covariance function has the form 
\begin{equation}
\mathbb{E}[B_{\beta,\alpha}(t)B_{\beta,\alpha}(s)]=\frac{1}{2\Gamma(\beta+1)}\big(t^{\alpha}+s^{\alpha}-|t-s|^{\alpha}\big),\quad t,s\geq0.\label{eq:auto-cv-gBm}
\end{equation}
\item For each $t,s\geq0$, the characteristic function of the increments
is 
\begin{equation}
\mathbb{E}\big[e^{i\theta(B_{\beta,\alpha}(t)-B_{\beta,\alpha}(s))}\big]=E_{\beta}\left(-\frac{\theta^{2}}{2}|t-s|^{\alpha}\right),\quad\theta\in\mathbb{R}.\label{eq:cf_gBm_increments}
\end{equation}
\item The process $B_{\beta,\alpha}$ is non Gaussian, $\alpha/2$-self-similar
with stationary increments.

\item The ggBm is not a semimartingale. Furthermore, $B_{\alpha,\beta}$
cannot be of finite variation in $[0,1]$ and by scaling and stationarity
of the increment on any interval in $\mathbb{R}^+$. 
\end{enumerate}

The ggBm admits different representations in terms of well-known processes. The most common is given in \cite{Mura_Pagnini_08} in the form 
\begin{equation}\label{eq:repr-Mura}
\big\{B_{\beta,\alpha}(t),\,t\geq0\big\}\overset{\mathcal{L}}{=}\big\{\sqrt{Y_{\beta}}B^{\alpha/2}(t),\,t\geq0\big\}.   
\end{equation}
Here, $\overset{\mathcal{L}}{=}$ means equality in law, the nonnegative random variable $Y_\beta$ has density $M_{\beta}$, called the $M$-Wright probability density
function, with the Laplace transform  
\begin{equation}\label{eq:M_wright}
\int_{0}^{\infty}e^{-s\tau}M_{\beta}(\tau)\,d\tau=E_{\beta}(-s),
\end{equation}
and $B^{\alpha/2}$ is a fBm independent of $Y_\beta$.
The generalized moments of the density $M_{\beta}$ of order $\delta>-1$ are finite and are given (cf.~\cite{Mura_Pagnini_08}) by
\begin{equation}\label{eq:Mbeta-moments}
   \int_{0}^{\infty}\tau^{\delta}M_{\beta}(\tau)\,d\tau=\frac{\Gamma(\delta+1)}{\Gamma(\beta\delta+1)}. 
\end{equation}
For our purposes, we give a more suitable representation of $B_{\beta,\alpha}$ as a subordination of fBm, which is essential in what follows. 

\begin{prop}\label{prop:subord-ggBm}The ggBm has the following representation
\begin{equation}\label{eq:repr-our}
\big\{B_{\beta,\alpha}(t),\,t\geq0\big\}\overset{\mathcal{L}}{=}\big\{B^{\alpha/2}(tY_{\beta}^{1/\alpha}),\,t\geq0\big\}.    
\end{equation}    
\end{prop}
\begin{proof}
We only need to show that the representations \eqref{eq:repr-Mura} and \eqref{eq:repr-our} have the same finite-dimensional distribution. For every $\theta=(\theta_1,\dots,\theta_n)\in\mathbb{R}^n$, we have
\begin{eqnarray*}
&&\mathbb{E}\left[\exp\left(\mathrm{i}\sum_{k=1}^{n}\theta_{k}B^{\alpha/2}(t_{k}Y_\beta^{1/\alpha})\right)\right] \\ &=&\int_0^\infty\mathbb{E}\left[\exp\left(\mathrm{i}\sum_{k=1}^{n}\theta_{k}B^{\alpha/2}(t_{k}y^{1/\alpha})\right)\right]M_\beta(y)\,\mathrm{d}y\\
&=&\int_0^\infty\mathbb{E}\left[\exp\left(\mathrm{i}\sum_{k=1}^{n}\theta_{k}y^{1/2}B^{\alpha/2}(t_{k})\right)\right]M_\beta(y)\,\mathrm{d}y\\
&=&\mathbb{E}\left[\exp\left(\mathrm{i}\sum_{k=1}^{n}\theta_{k}Y_\beta^{1/2}B^{\alpha/2}(t_{k})\right)\right].dn
\end{eqnarray*}
\end{proof}

\subsubsection{Canonical Realization}
\label{subsec:realization-Bm}
We will distinguish two classes from the family $B_{\beta,\alpha}$ depending on the parameter $\alpha$.

\vspace{0.1cm}

\noindent {\bf Case} $\alpha=1$.

\vspace{0.1cm}

The corresponding class was introduced by Schneider \cite{Schneider90a, Schneider90}. We denote it by $B_\beta:=B_{\beta,1}$, $0<\beta<1$, and call it the grey Brownian motion. It follows from \eqref{eq:repr-our} that $B_\beta$ is realized as the subordination of the Brownian motion by the process $\{tY_\beta, t\ge0\}$.
In the space of continuous functions, this realization is given below. 

First,  recall from Remark~\ref{rem:fBm}-6 that $W^{1/2}$ is a standard Brownian motion on the classical Wiener space $\big(\mathbb{W},\mathcal{B}(\mathbb{W}),\mathbb{P}_{1/2}\big)$.

Second, let $(\mathbb{R}^{+},\mathcal{B}(\mathbb{R}^{+}),\mathbb{P}_{Y_{\beta}})$
be the probability space where $\mathbb{P}_{Y_{\beta}}$ is the law
of the random variable $Y_{\beta}$. Then $Y_{\beta}$ is realized
on $\mathbb{R}^{+}$ as the identity map 
\begin{equation}\label{eq:crealization_Ybeta}
\mathcal{Y}_{\beta}:\mathbb{R}^{+}\longrightarrow\mathbb{R}^{+},\;\tau\mapsto\mathcal{Y}_\beta(\tau):=\tau.
\end{equation}
Since $Y_\beta$ and $B^{1/2}$ are independent, the grey Brownian motion $B_{\beta}$ can be realized as the canonical process $X_\beta$ defined on the product space $\big(\mathbb{W}\times\mathbb{R}^{+},\mathcal{B}(\mathbb{W})\otimes\mathcal{B}(\mathbb{R}^{+}),\mathbb{P}_{1/2}\otimes\mathbb{P}_{Y_{\beta}}\big)$
by 
\begin{equation}\label{eq:crealization_Xbeta}
X_{\beta}(t)(w,\tau):=W^{1/2}(t\mathcal{Y}_{\beta}(\tau))(w)=w(t\tau),\quad t\ge0,\quad w\in\mathbb{W},\;\tau\in\mathbb{R}^{+}.
\end{equation}
We denote its law by $\mu_\beta $.

\noindent {\bf Case} $\alpha \in (1,2]$.

\vspace{0.1cm}
Let $H:=\frac{\alpha}{2}>1/2$ and $W^H$ be the fractional Brownian motion on the space $\big(\mathbb{W},\mathcal{B}(\mathbb{W}),\mathbb{P}_{H}\big)$, see Remark~\ref{rem:fBm}-6. As $Y_{\beta}$
and $B^{H}$ are independent, the generalized grey Brownian motion $B_{\beta,2H}$ can be realized as the canonical process $X_{\beta,2H}$ defined on the product space $\big(\mathbb{W}\times\mathbb{R}^{+},\mathcal{B}(\mathbb{W})\otimes\mathcal{B}(\mathbb{R}^{+}),\mathbb{P}_{H}\otimes\mathbb{P}_{Y_{\beta}}\big)$
by 
\[
X_{\beta,2H}(t)(w,\tau):=W^{H}\big(t\mathcal{Y}_{\beta}^{1/(2H)}(\tau)\big)(w)=w\big(t\tau^{1/(2H)}\big),\; t,\tau\in\mathbb{R}^{+},\; w\in\mathbb{W}.
\]
In this case, the law of $X_{\beta,2H}$ is denoted by $\mu_{\beta,2H} $.

\section{The Cameron--Martin Theorem}
\label{sec:Cameron-Martin-formula}
\subsection{For Grey Brownian Motion}


For every $h\in\mathcal{H}_{1/2}$ we denote by $X_{\beta}^{h}$
the process defined on $\big(\mathbb{W}\times\mathbb{R}^{+},\mathcal{B}(\mathbb{W})\otimes\mathcal{B}(\mathbb{R}^{+}),\mathbb{P}_{1/2}\otimes\mathbb{P}_{Y_{\beta}}\big)$ by
\[
X_{\beta}^{h}(t):=X_\beta(t)+h(t\mathcal{Y}_{\beta})=W^{1/2}(t\mathcal{Y}_{\beta})+h(t\mathcal{Y}_{\beta}),\quad t\ge0,
\]
and its law is represented by $\mu_{\beta}^h$.

For $T>0$ we consider the processes $X_{\beta,T}:=\big\{X_{\beta}(t)\mid t\in[0,T]\big\}$ and $X_{\beta,T}^{h}:=\big\{X_{\beta}^{h}(t)\mid t\in[0,T]\big\}$ and their laws, denoted by $\mu_{\beta,T}$ and $\mu_{\beta,T}^h$, respectively.

The Cameron-Martin theorem for the grey Brownian motion $X_\beta$ is expressed as follows:
\begin{thm}
\label{thm:CM-gBm}
 Let  $h\in\mathcal{H}_{1/2}$ be given. Then we have:
\begin{enumerate}
\item For every $T>0$ the measures $\mu_{\beta,T}^{h}$ and $\mu_{\beta,T}$
are equivalent and the Radon-Nikodym density is given by
\begin{equation}\label{eq:Cameron-Martin-gBm0T}
\frac{\mathrm{d}\mu_{\beta,T}^{h}}{\mathrm{d}\mu_{\beta,T}}=\mathcal{E}\left(\int_{0}^{.}\dot{h}(s)\,\mathrm{d}W^{1/2}(s)\right)\!(T\mathcal{Y}_{\beta}).
\end{equation}
\item The measures $\mu_\beta^h$ and $\mu_\beta$ are equivalent and the Radon-Nikodym density is
\begin{equation}\label{eq:Cameron-Martin-gBm0inf}
\frac{\mathrm{d}{\mu_\beta^h}}{\mathrm{d}{\mu_\beta}}=\mathcal{E}\left(\int_{0}^{.}\dot{h}(s)\,\mathrm{d}W^{1/2}(s)\right)\!(\infty).
\end{equation}
\end{enumerate}
\end{thm}

\begin{rem}\label{rem:gBm-exp-martingale}
    Before proving the theorem, we will elucidate the two equations \eqref{eq:Cameron-Martin-gBm0T} and \eqref{eq:Cameron-Martin-gBm0inf}.
    \begin{enumerate}
    \item The expectation of $\mathcal{E}\left(\displaystyle{\int_{0}^{.}}\dot{h}(s)\,\mathrm{d}W^{1/2}(s)\right)\!(T\mathcal{Y}_{\beta})$ with respect to $\mathbb{P}_{1/2}\otimes\mathbb{P}_{Y_{\beta}}$ is equal to $1$, that is,
\[
\int_{\mathbb{W}\times\mathbb{R}^{+}}\mathcal{E}\left(\displaystyle{\int_{0}^{.}}\dot{h}(s)\,\mathrm{d}W^{1/2}(s)\right)(T\tau)\,\mathbb{P}_{1/2}(\mathrm{d}w)\,\mathbb{P}_{Y_{\beta}}(\mathrm{d}\tau)=1.
\]
Indeed it follows from Tonelli's theorem and the fact that the expectation of a martingale is constant in time, that is, for any $t\geq 0$ we have
\[
\int_{\mathbb{W}}\mathcal{E}\left(\displaystyle{\int_{0}^{.}}\dot{h}(s)\,\mathrm{d}W^{1/2}(s)\right)(t)\,\mathbb{P}_{1/2}(\mathrm{d}w)=1.
\]  
\item It follows from Remark~\ref{rem:exp-martinal-infty} that the right-hand side of \eqref{eq:Cameron-Martin-gBm0inf} is well-defined. On the other hand, since $\mathbb{P}_{Y_\beta}(\{0\})=0$ we have $\mathbb{P}_{Y_\beta}(\{\infty \mathcal{Y}_\beta=\infty\})=1$, thus the r.h.s.~of \eqref{eq:Cameron-Martin-gBm0inf} implicitly depends on $Y_\beta$ in the following sense:
\[
\mathcal{E}\left(\displaystyle{\int_{0}^{\cdot}}\dot{h}(s)\,\mathrm{d}W^{1/2}(s)\right)(\infty)=\mathcal{E}\left(\displaystyle{\int_{0}^{\cdot}}\dot{h}(s)\,\mathrm{d}W^{1/2}(s)\right)(\infty \mathcal{Y}_\beta),
\]
$\mathbb{P}_{1/2}\otimes\mathbb{P}_{Y_{\beta}}$ almost surely.
    \end{enumerate}

\end{rem}
\begin{proof}
1. Let $T>0$ be given. It is sufficient to prove
\begin{align*}
 & \int_{\mathbb{W}\times\mathbb{R}^{+}}f\big(w(t_{1}\tau)+h(t_{1}\tau),\ldots,w(t_{n}\tau)+h(t_{n}\tau)\big)\,\mathbb{P}_{1/2}(\mathrm{d}w)\,\mathbb{P}_{Y_{\beta}}(\mathrm{d}\tau)\\
 & =\int_{\mathbb{W}\times\mathbb{R}^{+}}f\big(w(t_{1}\tau),\ldots,w(t_{n}\tau)\big)\mathcal{E}\left(\int_{0}^{.}\dot{h}(s)\,\mathrm{d}W^{1/2}(s)\right)\!(T\tau)\,\mathbb{P}_{1/2}(\mathrm{d}w)\,\mathbb{P}_{Y_{\beta}}(\mathrm{d}\tau),
\end{align*}
where $n\in\mathbb{N}$, $0<t_{1}<\dots<t_{n}\le T$ and $f\in C_{b}(\mathbb{R}^{n})$ (the set of continuous bounded functions on $\mathbb{R}^n$), see \cite[Thm.~3.3]{Billingsley1995}.

For every fixed $\tau\in\mathbb{R}^+$, the classical Cameron--Martin formula applied to the bounded measurable functional on $\mathbb{W}_{T\tau}$
\[
w\mapsto f\left(w\left(t_{1}\tau\right),\ldots,w\left(t_{n}\tau\right)\right)
\]
yields 
\[
\begin{aligned} & \int_{\mathbb{W}}f\big(w(t_{1}\tau)+h(t_{1}\tau),\ldots,w(t_{n}\tau)+h(t_{n}\tau)\big)\,\mathbb{P}_{1/2}(\mathrm{d}w)\\
 & =\int_{\mathbb{W}_{T\tau}}f\big(w(t_{1}\tau)+h(t_{1}\tau),\ldots,w(t_{n}\tau)+h(t_{n}\tau)\big)\,\mathbb{P}_{1/2}(\mathrm{d}w)\\
 & =\int_{\mathbb{W}_{T\tau}}f\left(w(t_{1}\tau),\ldots,w(t_{n}\tau)\right)\mathcal{E}\left(\int_{0}^{.}\dot{h}(s)\,\mathrm{d}W^{1/2}(s)\right)\!(T\tau)\,\mathbb{P}_{1/2}(\mathrm{d}w)\\
 & =\int_{\mathbb{W}}f\left(w(t_{1}\tau),\ldots,w(t_{n}\tau)\right)\mathcal{E}\left(\int_{0}^{.}\dot{h}(s)\,\mathrm{d}W^{1/2}(s)\right)\!(T\tau)\,\mathbb{P}_{1/2}(\mathrm{d}w).
\end{aligned}
\] 
Therefore, we may integrate both sides of the above equality w.r.t $\mathbb{P}_{Y_{\beta}}$
to obtain the result.

2. It follows by an easy adaption of the above arguments with $0\le t_{1}<t_{2}<\dots<t_{n}<+\infty$
and the classical Cameron--Martin formula for functionals defined on
$\mathbb{W}$.
\end{proof}

\begin{thm}
\label{thm:equivalent-H}Let $h\in\mathbb{W}$ be given. If $\mu_\beta$
and $\mu_{\beta,h}$ are equivalent, then $h\in\mathcal{H}_{1/2}$.
\end{thm}

\begin{proof}
Let $h\in\mathbb{W}\backslash\mathcal{H}_{1/2}$ be given. It is well known that Gaussian
measures $\mathbb{P}_{W^{1/2}}$ and $\mathbb{P}_{W^{1/2}+h}$ are singular; see
Theorem 2.2 on page 339 in \cite{Revuz-Yor-94} (see also Lemma~3.12 in \cite{MR92}). Hence, there exists
a measurable set $A\subset\mathbb{W}$ such that
\[
\mathbb{P}_{W^{1/2}}(A)=1\qquad\mathrm{and}\qquad\mathbb{P}_{W^{1/2}+h}(A)=0.
\]
Define the subset of $\mathbb{W}$
\[
\tilde{A}:=\{w(\cdot\tau)\mid w\in A,\;\tau\in\mathbb{R}^{+}\}.
\]
It is not difficult to see that
\[
\mu_\beta(\tilde{A})=\mathbb{P}_{W^{1/2}}(A)=1
\]
and 
\[
\mu_{\beta,h}(\tilde{A})=\mathbb{P}_{W^{1/2}+h}(A)=0.\qedhere
\]
\end{proof}
A consequence of Theorem~\ref{thm:CM-gBm}-1 and Theorem~\ref{thm:equivalent-H}
is the following corollary.
\begin{cor}
Let $h\in\mathbb{W}$ be an absolutely continuous function such that $\int_{0}^{T}|\dot{h}(t)|^{2}\,\mathrm{d}t<\infty$ for any $T>0$ and
$\int_{0}^{\infty}|\dot{h}(t)|^{2}\,\mathrm{d}t=\infty$. Then, for any $T>0$ the measures
$\mu_{\beta,T}^{h}$ and $\mu_{\beta,T}$ are equivalent,
but the measures $\mu_{\beta,h}$ and $\mu_{\beta}$
are singular.
\end{cor}

\subsection{For Generalized Grey Brownian Motion}
\label{subsec:canonical-real-ggBm}

For every $h\in\mathcal{H}_{H}$ we define on the space $\big(\mathbb{W}\times\mathbb{R}^{+},\mathcal{B}(\mathbb{W})\otimes\mathcal{B}(\mathbb{R}^{+}),\mathbb{P}_{H}\otimes\mathbb{P}_{Y_{\beta}}\big)$ the process $X_{\beta,2H}^{h}$, for every $t\geq0$, by 
\[
X_{\beta,2H}^{h}(t)(w,\tau):=W^{H}(t\mathcal{Y}_{\beta}^{1/(2H)}(\tau))(w)+h(t\mathcal{Y}_{\beta}^{1/(2H)}(\tau))=w(t\tau^{1/(2H)})+h(t\tau^{1/(2H)}),
\]
and for any $T>0$  the process $X_{\beta,2H,T}^{h}$
\[
X_{\beta,2H,T}^{h}:=\big\{X_{\beta,2H}^{h}(t),\;t\in[0,T]\big\}.
\]
The corresponding laws are denoted by $\mu_{\beta,2H}^{h}$ and $\mu_{\beta,2H,T}^{h}$, respectively. We also denote by $\mu_{\beta,2H,T}$ the law of the process $X_{\beta,2H,T}:=\big\{X_{\beta,2H}(t),\; t\in[0,T]\big\}$.

Now we are ready to state the Cameron-Martin theorem for the generalized grey Brownian motion.
\begin{thm}\label{thm:CM-ggBm}
Let $h\in\mathcal{H}_{H}$ be given. Then we have: 
\begin{enumerate}
\item  for every $T>0$ the measures $\mu_{\beta,2H,T}^{h}$ and $\mu_{\beta,2H,T}$ 
are equivalent and the Radon-Nikodym density is given by
\begin{eqnarray}\label{eq:Cameron-Martin-ggBm0T}
\frac{\mathrm{d}\mu_{\beta,2H,T}^{h}}{\mathrm{d}\mu_{\beta,2H,T}}&=& \exp\left(\int_{0}^{t}{h}(s)\,\mathrm{d}W^H(s)-\frac{1}{2}\|h\|^2_{\mathcal{H}_H(t)}\right)\!\!\bigg|_{t=T\mathcal{Y}_\beta^{1/(2H)}}\nonumber \\
&=&\mathcal{E}\left(\int_{0}^{.}\dot{h}
(s)\,\mathrm{d}\widetilde{W}(s)\right)\!(T\mathcal{Y}_{\beta}^{1/(2H)}).
\end{eqnarray}
\item The measures $\mu_{\beta,2H}^h$ and $\mu_{\beta,2H}$ are equivalent, and the Radon-Nikodym density has the form

\begin{eqnarray}\label{eq:Cameron-Martin-ggBm-0inf}
\frac{\mathrm{d}\mu_{\beta,2H}^h}{\mu_{\beta,2H}}&=& \exp\left(\int_{0}^{\infty}{h}(s)\,\mathrm{d}W^H(s)-\frac{1}{2}\|h\|^2_{\mathcal{H}_H}\right) \nonumber \\
&=&\mathcal{E}\left(\int_{0}^{.}\dot{h}(s)\,\mathrm{d}\widetilde{W}(s)\right)\!(\infty).
\end{eqnarray}
\end{enumerate}
Here, $\widetilde{W}$ is the standard Brownian motion related to $W^H$ by Equation \eqref{eq:Bm-associated-fBm}.
\end{thm}
\begin{proof}
1. To keep the notation short, we set $\tau_H:=\tau^{1/(2H)}$.  We have to show 
\begin{eqnarray*}
 && \int_{\mathbb{W}\times\mathbb{R}^{+}}f\big(w(t_{1}\tau_H)+h(t_{1}\tau_H),\ldots,w(t_{n}\tau_H)+h(t_{n}\tau_H)\big)\,\mathbb{P}_{H}(\mathrm{d}w)\,\mathbb{P}_{Y_{\beta}}(\mathrm{d}\tau)\\
&=&\int_{\mathbb{W}\times\mathbb{R}^{+}}f\big(w(t_{1}\tau_H),\ldots,w(t_{n}\tau_H)\big) \mathcal{E}\left(\int_{0}^{.}\dot{h}(s)\,\mathrm{d}\widetilde{W}(s)\right)\!(T\tau_H)\,\mathbb{P}_{H}(\mathrm{d}w)\,\mathbb{P}_{Y_{\beta}}(\mathrm{d}\tau),
\end{eqnarray*}
where $n\in\mathbb{N}$, $0<t_{1}<\dots<t_{n}\le T$ and $f\in C_{b}(\mathbb{R}^{n})$.

Taking into account the Remark~\ref{rem:RKHS-finite-time}, applying the classical Cameron--Martin formula~\eqref{eq:Cameron-Martin-fBm} to the bounded measurable
functional on $\mathbb{W}_{T\tau_H}$, 
\[
w\mapsto f\big(w(t_{1}\tau_H),\ldots,w(t_{n}\tau_H)\big)
\]
yields 
\[
\begin{aligned} & \int_{\mathbb{W}}f\big(w(t_{1}\tau_H)+h(t_{1}\tau_H),\ldots,w(t_{n}\tau_H)+h(t_{n}\tau_H)\big)\,\mathbb{P}_{H}(\mathrm{d}w)\\
 & =\int_{\mathbb{W}_{T\tau_H}}f\big(w(t_{1}\tau_H)+h(t_{1}\tau_H),\ldots,w(t_{n}\tau_H)+h(t_{n}\tau_H)\big)\,\mathbb{P}_{H}(\mathrm{d}w)\\
 & =\int_{\mathbb{W}_{T\tau_H}}f\big(w(t_{1}\tau_H),\ldots,w(t_{n}\tau_H)\big)\mathcal{E}\left(\int_{0}^{.}\dot{h}(s)\,\mathrm{d}\widetilde{W}(s)\right)\!(T\tau_H)\,\mathbb{P}_{H}(\mathrm{d}w)\\
 & =\int_{\mathbb{W}}f\big(w(t_{1}\tau_H),\ldots,w(t_{n}\tau_H)\big)\mathcal{E}\left(\int_{0}^{.}\dot{h}(s)\,\mathrm{d}\widetilde{W}(s)\right)\!(T\tau_H)\,\mathbb{P}_{H}(\mathrm{d}w).
\end{aligned}
\]
Therefore, we may integrate both sides of the above equality w.r.t.~$\mathbb{P}_{Y_{\beta}}$
to obtain the result.

2. It follows by an easy adaption of the above arguments with $0\le t_{1}<t_{2}<\dots<t_{n}<\infty$
and the Cameron--Martin formula \eqref{eq:Cameron-Martin-fBm} for functionals defined in
$\mathbb{W}$.
\end{proof}

\begin{rem}
    Note that the results of Theorem~\ref{thm:CM-gBm}--1 and Theorem~\ref{thm:CM-ggBm}--1 are still valid if we assume that $h\in \mathcal{H}_H(T)$ for any $T>0$.   
\end{rem}

\begin{thm}
\label{thm:equivalent-2H}Let $h\in\mathbb{W}$ be given. If $\mu_{\beta,2H}$
and $\mu_{\beta,2H}^h$ are equivalent, then $h\in\mathcal{H}_{H}$.
\end{thm}
\begin{proof}
Let $h\in\mathbb{W}\backslash\mathcal{H}_{H}$ be given. It follows from Proposition~20 in \cite{Coutin2007} 
that the Gaussian
measures $\mathbb{P}_{W^{H}}$ and $\mathbb{P}_{W^{H}+h}$ are singular. Therefore, there exists
a measurable set $A_H\subset\mathbb{W}$ such that
\[
\mathbb{P}_{W^{H}}(A_H)=1\qquad\mathrm{and}\qquad\mathbb{P}_{W^{H}+h}(A_H)=0.
\]
Define the subset of $\mathbb{W}$
\[
\tilde{A}_H:=\{w(\cdot\tau_H)\mid w\in A_H,\;\tau\in\mathbb{R}^{+}\}.
\]
It is not difficult to see that 
\[
\mu_{\beta,2H}(\tilde{A}_H)=\mathbb{P}_{W^{H}}(A_H)=1
\]
and 
\[
\mu_{\beta,2H}^{h}(\tilde{A}_H)=\mathbb{P}_{W^{H}+h}(A_H)=0.\qedhere
\]
\end{proof}
\begin{rem} We would like to specify the support of the laws: $\mu_\beta$, $\mu_{\beta,2H}$, $\mu_{\beta,T}$, and $\mu_{\beta,2H,T}$.
\begin{eqnarray*}
\mathrm{supp}(\mu_\beta)= \Omega^{1/2}&:=&\big\{ w\left(\cdot\tau\right): [0,\infty[\longrightarrow\mathbb{R},\,t\mapsto w(t\tau)\mid w\in\mathbb{W},\tau\geq0\big\}.\\
\mathrm{supp}(\mu_{\beta,2H})=\Omega^H&:=&\left\{ w\left(\cdot\tau^{1/(2H)}\right)\mid w\in\mathbb{W},\tau\geq0\right\}.\\
\mathrm{supp}(\mu_{\beta,T})=\Omega^{1/2}_T&:=&\left\{ \tilde{w}\!\!\upharpoonright_{\left[0,T\right]}\,\mid\tilde{w}\in\Omega^{1/2}\right\}.\\
\mathrm{supp}(\mu_{\beta,2H,T})=\Omega_T^H&:=&\left\{ \tilde{w}\!\!\upharpoonright_{\left[0,T\right]}\,\mid\tilde{w}\in\Omega^H\right\}.
\end{eqnarray*}
\end{rem}

\section{Integration by Parts Formula and Closability of the Derivative Operator}\label{sec:IbP-closability}
As an application of the Cameron--Martin formulas from Section~\ref{sec:Cameron-Martin-formula}, we first introduce the directional derivative in the directions of the elements $h\in\mathcal{H}_{H}$ for smooth functionals. We show an integration by parts formula leading to the closability of the directional derivative on $L^p(\mathbb{P}_H\otimes\mathbb{P}_{Y_\beta})$, $p\ge1$.

We begin by analyzing the class of gBm $X_{\beta,T}$, $T>0$, (Subsection~\ref{subsec:gradient-gBm}) and then move on to the broader class $X_{\beta,2H,T}$, $T>0$, of ggBm from Subsection~\ref{subsec:gradient-ggBm}.

\subsection{Derivative Operator for Grey Brownian Motion}\label{subsec:gradient-gBm}
We start by introducing the directional derivative in special directions (from the Cameron--Marin space $\mathcal{H}_{1/2}$) of a real-valued measurable function $F$ on $\big(\mathbb{W}\times\mathbb{R}^{+},\mathcal{B}(\mathbb{W})\otimes\mathcal{B}(\mathbb{R}^{+}),\mathbb{P}_{1/2}\otimes\mathbb{P}_{Y_{\beta}}\big)$. 

Given $T>0$, $h\in \mathcal{H}_{1/2}$, and $F:\mathbb{W}\times\mathbb{R}^+\longrightarrow\mathbb{R}$ the partial derivative of $F$ in the direction $h$ is defined by
\[
(\partial_{1,h}F)(w,\tau):=\lim_{\varepsilon\to0}\frac{F(w+\varepsilon h,\tau)-F(w,\tau)}{\varepsilon},\quad w\in\mathbb{W}, \tau\in\mathbb{R}^+,
\]
whenever the limit exists. 

We consider the class $\mathcal{F}C_b^\infty$ of functions $F$ of the form
\begin{equation}\label{eq:smooth-cylinder-ftion}
F\big(w,\tau\big):=f\big(w(t_{1}\tau),\ldots,w( t_{n}\tau)\big),\quad w\in\mathbb{W},\tau\in\mathbb{R}^+,
\end{equation}
where $n\in\mathbb{N}$, $0<t_{1}<\dots<t_{n}\le T$ and $f\in C_{b}^\infty(\mathbb{R}^{n})$ (the set of $C^\infty$ bounded real-valued functions on $\mathbb{R}^n$).  Note that this class is closely related to the process $X_{\beta,T}$. In fact, any $F\in \mathcal{F}C_{b}^{\infty}$ of the form \eqref{eq:smooth-cylinder-ftion} is simply
\begin{equation}\label{eq:cylinder-gBm}
F(w,\tau)=f\big(X_{\beta,T} (t_1)(w,\tau),\dots ,X_{\beta,T} (t_n)(w,\tau)\big),
\end{equation}
generally referred to as smooth cylindrical function. It is simple to verify that the derivative $\partial_{1,h}F$ exists for any $F\in\mathcal{F}C_b^\infty$ and we have
\begin{equation}\label{eq:derivative-cylinder-fction}
    (\partial_{1,h}F)(w,\tau)=\sum_{i=1}^n\partial_{i} f\big(w(t_1\tau),\dots,w( t_n\tau)\big)h(t_i\tau),\quad w\in\mathbb{W},\tau\in\mathbb{R}^+,
\end{equation}
where $\partial_{i}$ denotes the partial derivative w.r.t.~$x_i$.

In what follows, we show that $\partial_{1,h}F$ is represented by an element
of the Cameron-Martin space $\mathcal{H}_{1/2}$, denoted by $\nabla_1 F$, where the subscript $1$ is used to indicate w.r.t.~the first variable $w$.
\begin{prop}
\label{prop:cont-derivative}Let $F\in\mathcal{F}C_{b}^{\infty}$ be given. Then, for any $w\in\mathbb{W}$,
$\tau\in\mathbb{R}^+$, 
\begin{enumerate}
\item the map
\[
\mathcal{H}_{1/2}\ni h\mapsto(\partial_{1,h}F)(w,\tau)\in \mathbb{R}
\]
is a bounded linear functional on $\mathcal{H}_{1/2}$,
\item there exists a unique element $\nabla_1 F(w,\tau)\in\mathcal{H}_{1/2}$
given by 
\begin{equation}\label{eq:gradient-repretion}
\big(\nabla_1 F(w,\tau)\big)(t)=\sum_{i=1}^{n}\partial_{i}f\big(w(t_{1}\tau),\dots,w(t_{n}\tau)\big)(t\wedge t_{i}\tau),\quad t\ge0
\end{equation}
and 
\begin{equation}\label{eq:gradient-repretion-dot}
\dot{\nabla}_1 F(w,\tau)(t)=\sum_{i=1}^{n}\partial_{i}f\big(w(t_{1}\tau),\dots,w(t_{n}\tau)\big)\mathbbm{1}_{[0,t_{i}\tau)}(t)
\end{equation}
satisfying 
\begin{equation}
(\partial_{1,h}F)(w,\tau)=\big(\nabla_1 F(w,\tau),h\big)_{\mathcal{H}_{1/2}}.\label{eq:derivative-repretion}
\end{equation}
\end{enumerate}
\end{prop}

\begin{proof}
1. Let $F\in\mathcal{F}C_{b}^{\infty}$
be given as in (\ref{eq:smooth-cylinder-ftion}). The linearity of
the map is obvious. For every $h\in\mathcal{H}_{1/2}$, using the
representation (\ref{eq:1/2-rep}) we have
\begin{eqnarray*}
|(\partial_{1,h}F)(w,\tau)| & = & \left|\sum_{i=1}^{n}\partial_{i}f\big(w(t_{1}\tau),\dots,w(t_{n}\tau)\big)h(t_{i}\tau)\right|\\
 & \le & \sum_{i=1}^{n}\|\partial_{i}f\|_{\infty}\left|\int_{0}^{t_{i}\tau}\dot{h}(s)\,\mathrm{d}s\right|.
\end{eqnarray*}
An application of the Cauchy-Schwarz inequality and the fact that
$\|\dot{h}\|_{2}=\|h\|_{\mathcal{H}_{1/2}}$, yields
\begin{equation}
|(\partial_{1,h}F)(w,\tau)|\le\sum_{i=1}^{n}\|\partial_{i}f\|_{\infty}\sqrt{t_{i}\tau}\|\dot{h}\|_{2}=\sum_{i=1}^{n}\|\partial_{i}f\|_{\infty}\sqrt{t_{i}\tau}\|h\|_{\mathcal{H}_{1/2}}.\label{eq:est-CM-norm}
\end{equation}
2. The equality (\ref{eq:derivative-repretion}) follows from 1.~and
the Riesz representation theorem. From (\ref{eq:derivative-cylinder-fction})
and (\ref{eq:1/2-rep}) we obtain 
\begin{eqnarray*}
(\partial_{1,h}F)(w,\tau) & = & \sum_{i=1}^{n}\partial_{i}f\big(w(t_{1}\tau),\dots,w(t_{n}\tau)\big)h(t_{i}\tau)\\
 & = & \sum_{i=1}^{n}\partial_{i}f\big(w(t_{1}\tau),\dots,w(t_{n}\tau)\big)\int_{0}^{t_{i}\tau}\dot{h}(s)\,\mathrm{d}s\\
 & = & \int_{0}^{\infty}\sum_{i=1}^{n}\partial_{i}f\big(w(t_{1}\tau),\dots,w(t_{n}\tau)\big)\mathbbm{1}_{[0,t_{i}\tau)}(s)\dot{h}(s)\,\mathrm{d}s.
\end{eqnarray*}
The equalities (\ref{eq:gradient-repretion}) and (\ref{eq:gradient-repretion-dot})
easily follow.
\end{proof}
Next, we show that the directional derivative $\partial_{1,h}F$ with  $F\in\mathcal{F}C_{b}^{\infty}$
and $h\in\mathcal{H}_{1/2}$, is an element of $L^{p}(\mathbb{P}_{1/2}\otimes\mathbb{P}_{Y_\beta})$
for any $p\ge1$. 
\begin{prop}
\label{prop:direct-deriv-Lp}Let $F\in\mathcal{F}C_{b}^{\infty}$
and $p\ge1$ be given. Then
\begin{enumerate}
\item $\partial_{1,h}F\in L^{p}(\mathbb{W}\times\mathbb{R}^+,\mathcal{F}^{X_{\beta,T}},\mathbb{P}_{1/2}\otimes\mathbb{P}_{Y_\beta})$, for any $h\in\mathcal{H}_{1/2}$,
where $\mathcal{F}^{X_{\beta,T}}:=\sigma\big(X_{\beta,T}(t),t\in\left[0,T\right]\big)$ is the natural $\sigma$-algebra generated by the process $X_\beta$ up to time $T$.
\item $\nabla_1 F\in L_{\mathcal{H}_{1/2}}^{p}(\mathcal{F}^{X_{\beta,T}},\mathbb{P}_{1/2}\otimes\mathbb{P}_{Y_\beta}):=L^{p}(\mathbb{W}\times\mathbb{R}^+,\mathcal{F}^{X_{\beta,T}},\mathbb{P}_{1/2}\otimes\mathbb{P}_{Y_\beta};\mathcal{H}_{1/2})$. 
\end{enumerate}
\end{prop}

\begin{proof}
1. Let $F\in\mathcal{F}C_{b}^{\infty}$
be given as in (\ref{eq:smooth-cylinder-ftion}). It follows from \eqref{eq:cylinder-gBm} and \eqref{eq:derivative-cylinder-fction} that $F$ and $\partial_{1,h}F$ are $\mathcal{F}^{X_{\beta,T}}$-measurable.  From the estimate~(\ref{eq:est-CM-norm})
we derive
\begin{eqnarray*}
\|\partial_{1,h}F\|_{p}^{p} & = & \int_{\mathbb{W}\times\mathbb{R}^{+}}|(\partial_{1,h}F)(w,\tau)|^{p}\,\mathrm{d}(\mathbb{P}_{1/2}\otimes\mathbb{P}_{Y_{\beta}})(w,\tau)\\
 & \le & \int_{\mathbb{W}\times\mathbb{R}^{+}}\left|\sum_{i=1}^{n}\|\partial_{i}f\|_{\infty}\sqrt{t_{i}\tau}\|h\|_{\mathcal{H}_{1/2}}\right|^{p}\mathrm{d}(\mathbb{P}_{1/2}\otimes\mathbb{P}_{Y_{\beta}})(w,\tau)\\
 & \le & n^{p-1}\|h\|_{\mathcal{H}_{1/2}}^{p}\sum_{i=1}^{n}\|\partial_{i}f\|_{\infty}^{p}t_{i}^{p/2}\int_{0}^{\infty}\tau^{p/2}\,\mathrm{d}\mathbb{P}_{Y_{\beta}}(\tau).\\
 & = & n^{p-1}\|h\|_{\mathcal{H}_{1/2}}^{p}\sum_{i=1}^{n}\|\partial_{i}f\|_{\infty}^{p}t_{i}^{p/2}\frac{\Gamma(p/2+1)}{\Gamma(\beta p/2+1)}<\infty.
\end{eqnarray*}
The second inequality is the convexity inequality and the last equality
follows from (\ref{eq:Mbeta-moments}).

\noindent 2. The $\mathcal{F}^{X_{\beta,T}}$-measurability of $\nabla_1 F$  and $\dot{\nabla}_1 F$ can be easily derived from \eqref{eq:gradient-repretion} and \eqref{eq:gradient-repretion-dot}. From (\ref{eq:gradient-repretion-dot}) and again the convexity inequality
we infer 
\begin{eqnarray*}
\|\nabla_1 F(w,\tau)\|_{\mathcal{H}_{1/2}}^{2} & = & \int_{0}^{\infty}|\dot{\nabla}_1 F(w,\tau)(t)|^{2}\,\mathrm{d}t\\
 & = & \int_{0}^{\infty}\left|\sum_{i=1}^{n}\partial_{i}f\big(w(t_{1}\tau),\dots,w(t_{n}\tau)\big)\mathbbm{1}_{[0,t_{i}\tau)}(t)\right|^{2}\mathrm{d}t\\
 & \le & n\sum_{i=1}^{n}t_{i}\tau\left(\partial_{i}f\big(w(t_{1}\tau),\dots,w(t_{n}\tau)\big)\right)^{2}.
\end{eqnarray*}
This implies
\begin{eqnarray*}
\|\nabla_1 F(w,\tau)\|_{\mathcal{H}_{1/2}}^{p} & \le & \left(n\sum_{i=1}^{n}t_{i}\tau\left(\partial_{i}f\big(w(t_{1}\tau),\dots,w(t_{n}\tau)\big)\right)^{2}\right)^{p/2}\\
 & \le & n^{p/2\vee(p-1)}\sum_{i=1}^{n}(t_{i}\tau)^{p/2}\|\partial_{i}f\|_\infty^p.
\end{eqnarray*}
Finally, the norm of $\|\nabla_1 F\|_{L_{\mathcal{H}_{1/2}}^{p}(\mathcal{F}^{X_{\beta,T}},\mathbb{P}_{1/2}\otimes\mathbb{P}_{Y_{\beta}})}$
is computed as
\begin{eqnarray*}
\|\nabla_1 F\|_{L_{\mathcal{H}_{1/2}}^{p}(\mathcal{F}^{X_{\beta,T}},\mathbb{P}_{1/2}\otimes\mathbb{P}_{Y_{\beta}})}^{p} 
 & = & \int_{\mathbb{W}\times\mathbb{R}^{+}}\|(\nabla_1 F)(w,\tau)\|_{\mathcal{H}_{1/2}}^{p}\,\mathrm{d}(\mathbb{P}_{1/2}\otimes \mathbb{P}_{Y_{\beta}})(w,\tau)\\
 & \le & n^{p/2\vee(p-1)}\int_{\mathbb{W}\times\mathbb{R}^{+}}\sum_{i=1}^{n}(t_{i}\tau)^{p/2}\|\partial_{i}f\|_\infty^p\,\mathrm{d}(\mathbb{P}_{1/2}\otimes \mathbb{P}_{Y_{\beta}})(w,\tau)\\
 & \le & n^{p/2\vee(p-1)}\sum_{i=1}^{n}\|\partial_{i}f\|_{\infty}^{p}t_{i}^{p/2}\int_{\mathbb{R}^{+}}\tau^{p/2}\mathrm{d}\mathbb{P}_{Y_{\beta}}(\tau)\\
 & = & n^{p/2\vee(p-1)}\sum_{i=1}^{n}\|\partial_{i}f\|_{\infty}^{p}t_{i}^{p/2}\frac{\Gamma(p/2+1)}{\Gamma(\beta p/2+1)}<\infty.
\end{eqnarray*}
This completes the proof.
\end{proof}
To investigate the closability of the directional derivative operator $\partial_{1,h}$ and $\nabla_1$ on $L^p(\mathbb{P}_H\otimes\mathbb{P}_{Y_\beta})$, $p\ge1$, we need the integration by parts formula for the gBm $X_{\beta,T}$, $T>0$. Below the expectation is taken w.r.t.~the probability measure $\mathbb{P}_{1/2}\otimes\mathbb{P}_{Y_\beta}$.

\begin{thm}\label{thm:IbP-gBm}
    Let $T>0$ be given. For any $h\in\mathcal{H}_{1/2}$ and $F,G\in\mathcal{F}C_b^\infty$, we have
    \begin{equation}\label{eq:IbP-gBm}
    \mathbb{E}\big[G\partial_{1,h}F\big]=\mathbb{E}\big[F\partial_{1,h}^*G\big],
    \end{equation}
    where 
    \[
    \partial_{1,h}^*G=-\partial_{1,h}G+G \left( \int_0^{\cdot} \dot{h}(t)\,\mathrm{d}W^{1/2}(t) \right)(T\mathcal{Y}_\beta).
    \]
\end{thm}

\begin{proof} 
    Let $T>0$ and $F,G\in\mathcal{F}C_b^\infty$ be given. The Cameron-Martin formula (see Eq.~\eqref{eq:Cameron-Martin-gBm0T}) for the grey Brownian motion says that, for every $\varepsilon\in\mathbb{R}$, we have
\begin{align*}
 & \int_{\mathbb{W}\times\mathbb{R}^{+}}F(w+\varepsilon h,\tau)\, G(w+\varepsilon h,\tau)\,\mathbb{P}_{1/2}(\mathrm{d}w)\,\mathbb{P}_{Y_{\beta}}(\mathrm{d}\tau)\\
 & =\int_{\mathbb{W}\times\mathbb{R}^{+}}F(w,\tau)\, G(w,\tau)\, \mathcal{E}\left(\int_{0}^{.}\varepsilon \dot{h}(s)\,\mathrm{d}W^{1/2}(s)\right)\!(T\tau)\,\mathbb{P}_{1/2}(\mathrm{d}w)\,\mathbb{P}_{Y_{\beta}}(\mathrm{d}\tau).
\end{align*}
We differentiate this equality w.r.t.~$\varepsilon$ and set $\varepsilon=0$. The assumptions on $h$ and $F$ allow us
to interchange the operations of differentiation and integration by using
the dominated convergence theorem, and we obtain the following
\begin{multline*}
\int_{\mathbb{W}\times\mathbb{R}^{+}}\partial_{1,h}F(w,\tau)\,G(w,\tau)\,\mathbb{P}_{1/2}(\mathrm{d}w)\,\mathbb{P}_{Y_{\beta}}(\mathrm{d}\tau) \\ +\int_{\mathbb{W}\times\mathbb{R}^{+}}F(w,\tau)\,\partial_{1,h}G(w,\tau)\,\mathbb{P}_{1/2}(\mathrm{d}w)\,\mathbb{P}_{Y_{\beta}}(\mathrm{d}\tau)\\ =\int_{\mathbb{W}\times\mathbb{R}^{+}}F(w,\tau)\,G(w,\tau)\left(\int_{0}^{T\tau}\dot{h}(s)\,\mathrm{d}W^{1/2}(s)\right)\,\mathbb{P}_{1/2}(\mathrm{d}w)\,\mathbb{P}_{Y_{\beta}}(\mathrm{d}\tau).
\end{multline*}
Note that by Lemma \ref{lem:h-integral-PYbeta}, all the terms in the above equality are well-defined. This concludes the proof. 
\end{proof}

From the identity \eqref{eq:IbP-gBm} we obtain the following theorem.
\begin{thm}\label{thm:closable-direct-derivative}
For every $p,q\geq 1$ and $h\in\mathcal{H}_{1/2}$
the derivative operator $\partial_{1,h}:\mathcal{F}C_{b}^{\infty}\longrightarrow L^{p}(\mathbb{W}\times\mathbb{R}^+,\mathcal{F}^{X_{\beta,T}},\mathbb{P}_{1/2}\otimes\mathbb{P}_{Y_\beta})$
is closable on $L^{q}(\mathbb{W}\times\mathbb{R}^{+},\mathcal{F}^{X_{\beta,T}},\mathbb{P}_{1/2}\otimes\mathbb{P}_{Y_\beta})$.  
\end{thm}

\begin{proof}
Let $p,q\ge1$ be given. We have to show that if $(F_{n})_{n\in\mathbb{N}}\subset\mathcal{F}C_{b}^{\infty}$
is a sequence such that $F_{n}\longrightarrow0$ in $L^{q}(\mathcal{F}^{X_{\beta,T}},\mathbb{P}_{1/2}\otimes\mathbb{P}_{Y_\beta})$
and $\partial_{1,h}F_{n}\longrightarrow Z$ in $L^{p}(\mathcal{F}^{X_{\beta,T}},\mathbb{P}_{1/2}\otimes\mathbb{P}_{Y_\beta})$, then
$Z=0$ in $L^{p}(\mathcal{F}^{X_{\beta,T}},\mathbb{P}_{1/2}\otimes\mathbb{P}_{Y_\beta})$.

Let $G\in\mathcal{F}C_{b}^{\infty}$
be fixed and $r\ge1$. According to Proposition~\ref{prop:direct-deriv-Lp}-1
we have $\partial_{1,h}G\in L^{r}(\mathcal{F}^{X_{\beta,T}},\mathbb{P}_{1/2}\otimes\mathbb{P}_{Y_\beta})$. In addition, using the Burkholder--Davis--Gundy
inequality, there exists a constant dependent  only in $r$ such that  
\begin{eqnarray*}
 & & \int_{0}^{\infty}\int_{\mathbb{W}}\left|\left(\int_{0}^{T\tau}\dot{h}(t)\,\mathrm{d}W^{1/2}(t)\right)(w)\right|^{r}\,\mathrm{d}\mathbb{P}_{1/2}(w)\,\mathrm{d}\mathbb{P}_{Y_{\beta}}(\tau)\\
 & \le & C_{r}\int_{0}^{\infty}\left(\int_{0}^{T\tau}|\dot{h}(t)|^{2}\,\mathrm{d}t\right)^{r/2}\,\mathrm{d}\mathbb{P}_{Y_{\beta}}(\tau)\\
 & \le & C_{r}\|\dot{h}\|_{2}^{r/2}=C_{r}\|h\|_{\mathcal{H}_{1/2}}^{r/2}.
\end{eqnarray*}
It is well known that the Wiener integral $\int_{0}^{\cdot}\dot{h}(t)\,\mathrm{d}W^{1/2}(t)$ can be obtained as the limit in probability of the Riemann sums; see \cite{Revuz-Yor-94}. Consequently, $\left(\int_{0}^{\cdot}\dot{h}(t)\,\mathrm{d}W^{1/2}(t)\right)(T\mathcal{Y}_{\beta})$ is $\mathcal{F}^{X_{\beta,T}}$-measurable.
Thus, 
\[
\partial_{1,h}^{*}G=-\partial_{1,h}G+G\left(\int_{0}^{\cdot}\dot{h}(t)\,\mathrm{d}W^{1/2}(t)\right)(T\mathcal{Y}_{\beta})\in L^{r}(\mathcal{F}^{X_{\beta,T}},\mathbb{P}_{1/2}\otimes\mathbb{P}_{Y_\beta}).
\]
Therefore, the integration by parts (\ref{eq:IbP-gBm}) and the fact that $\partial_{1,h}^*G\in L^{q'}(\mathcal{F}^{X_{\beta,T}},\mathbb{P}_{1/2}\otimes\mathbb{P}_{Y_\beta})$, $q'$ is the conjugate exponent of $q$, yields 
\[
\mathbb{E}[GZ]=\lim_{n\to\infty}\mathbb{E}[G\partial_{1,h}F_{n}]=\lim_{n\to\infty}\mathbb{E}[F_{n}\partial_{1,h}^{*}G]=0.
\]
We can deduce that $Z=0$ from the facts that the $\sigma$-algebra $\mathcal{F}^{X_{\beta,T}}$ is generated by the elements of $\mathcal{F}C_{b}^{\infty}$ and the density of  $\mathcal{F}C_{b}^{\infty}$ in $L^{p}(\mathcal{F}^{X_{\beta,T}},\mathbb{P}_{1/2}\otimes\mathbb{P}_{Y_\beta})$. 
\end{proof}
The closability of the operator $\nabla_1$ on $L^p(\mathbb{P}_H\otimes\mathbb{P}_{Y_\beta})$ is based on that of $\partial_{1,h}$. This is the subject of the following theorem.
\begin{thm}\label{thm:clasable-gradient}
For every $p,q\geq 1$, the operator $\nabla_1:\mathcal{F}C_{b}^{\infty}\longrightarrow L_{\mathcal{H}_{1/2}}^{p}(\mathcal{F}^{X_{\beta,T}},\mathbb{P}_{1/2}\otimes\mathbb{P}_{Y_\beta})$
is closable in $L^{q}(\mathbb{W}\times\mathbb{R}^{+},\mathcal{F}^{X_{\beta,T}},\mathbb{P}_{1/2}\otimes\mathbb{P}_{Y_\beta})$.
\end{thm}

\begin{proof}
Let $p,q\ge1$ and $(F_{n})_{n\in\mathbb{N}}\subset\mathcal{F}C_{b}^{\infty}$
be a sequence such that $F_{n}\longrightarrow0$ in $L^{q}(\mathcal{F}^{X_{\beta,T}},\mathbb{P}_{1/2}\otimes\mathbb{P}_{Y_\beta})$
and $\nabla_1 F_{n}\longrightarrow Z$ in $L_{\mathcal{H}_{1/2}}^{p}(\mathcal{F}^{X_{\beta,T}},\mathbb{P}_{1/2}\otimes\mathbb{P}_{Y_\beta})$.
We have to show that $Z=0$ in $L_{\mathcal{H}_{1/2}}^{p}(\mathcal{F}^{X_{\beta,T}},\mathbb{P}_{1/2}\otimes\mathbb{P}_{Y_\beta})$.
First, notice that for any $h\in\mathcal{H}_{1/2}$ we have
\[
\partial_{1,h}F_{n}=(\nabla_1 F_{n},h)_{\mathcal{H}_{1/2}}\longrightarrow(Z,h)_{\mathcal{H}_{1/2}},\;\mathrm{in}\;L^{p}(\mathcal{F}^{X_{\beta,T}},\mathbb{P}_{1/2}\otimes\mathbb{P}_{Y_\beta})\;\mathrm{as}\;n\to\infty.
\]
Since $\partial_{1,h}$ is closable (cf.~Theorem~\ref{thm:closable-direct-derivative}),
then $(Z,h)_{\mathcal{H}_{1/2}}=0$ in $L^{p}(\mathcal{F}^{X_{\beta,T}},\mathbb{P}_{1/2}\otimes\mathbb{P}_{Y_\beta})$. 
Hence,
\[
\mathbb{E}[G(Z,h)_{\mathcal{H}_{1/2}}]=0,\quad\forall G\in\mathcal{F}C_{b}^{\infty}.
\]
By linearity we obtain 
\[
\mathbb{E}\left[\left(Z,\sum_{i=1}^n G_ih_i \right)_{\mathcal{H}_{1/2}}\right]=0,
\]
for any element of the set 
\[
S(\mathcal{H}_{1/2}):=\left\{ \sum_{i=1}^{n}G_{i}h_{i}\,\middle|\, n\in\mathbb{N},\, G_{i}\in\mathcal{F}C_{b}^{\infty}\right\}, 
\]
where $\{h_{i}\in\mathcal{H}_{1/2},\;i\in\mathbb{N}\}$ is an orthonormal
basis of $\mathcal{H}_{1/2}$. Since $S(\mathcal{H}_{1/2})$ is dense in $L_{\mathcal{H}_{1/2}}^{p}(\mathcal{F}^{X_{\beta,T}},\mathbb{P}_{1/2}\otimes\mathbb{P}_{Y_\beta})$, 
we conclude that $Z=0$ in $L_{\mathcal{H}_{1/2}}^{p}(\mathcal{F}^{X_{\beta,T}},\mathbb{P}_{1/2}\otimes\mathbb{P}_{Y_\beta})$.
\end{proof}

\subsection{Derivative Operator for Generalized Grey Brownian Motion}\label{subsec:gradient-ggBm}

We intend to broaden the findings from Subsection~\ref{subsec:gradient-gBm} to the class of processes $X_{\beta,2H,T}$, $T>0$, introduced in Subsection~\ref{subsec:canonical-real-ggBm}.

We fix $H\in(1/2,1)$, $T>0$, and consider the class $\mathcal{F}C_{b,H}^\infty$ of measurable functions on $\big(\mathbb{W}\times\mathbb{R}^{+},\mathcal{B}(\mathbb{W})\otimes\mathcal{B}(\mathbb{R}^{+}),\mathbb{P}_{H}\otimes\mathbb{P}_{Y_{\beta}}\big)$ of the form
\begin{equation}\label{eq:smooth-cylinder-ftionH}
F\big(w,\tau\big):=f\big(w(t_{1}\tau_H),\ldots,w( t_{n}\tau_H)\big),\quad w\in\mathbb{W},\tau\in\mathbb{R}^+,
\end{equation}
where $n\in\mathbb{N}$, $0<t_{1}<\dots<t_{n}\le T$, $\tau_H=\tau^{1/(2H)}$, and $f\in C_{b}^\infty(\mathbb{R}^{n})$.
Note that each $F\in \mathcal{F}C_{b,H}^{\infty}$ of the form \eqref{eq:smooth-cylinder-ftionH} is represented as 
\begin{equation}\label{eq:cylinder-ggBm}
F(w,\tau)=f\big(X_{\beta,2H,T} (t_1)(w,\tau),\dots ,X_{\beta,2H,T} (t_n)(w,\tau)\big).
\end{equation} Moreover, we can easily see that for every $F\in \mathcal{F}C_{b,H}^\infty$ and $h\in \mathcal{H}_{H}$ the partial derivative $\partial_{1,h}F$ in the direction $h$ exists and we have
\begin{equation}\label{eq:derivative-cylinder-fction-ggBm}
    \partial_{1,h}F(w,\tau)=\sum_{i=1}^n\partial_{i} f\big(w(t_1\tau_H),\dots,w( t_n\tau_H)\big)h(t_i\tau_H), \quad w\in\mathbb{W},\tau\geq0. 
\end{equation}

The next proposition shows the two basic properties of the partial derivative operator $\partial_{1,h}$, $h\in\mathcal{H}_H$.
\begin{prop}
\label{prop:cont-derivative-ggBm}Let $F\in\mathcal{F}C_{b,H}^\infty$
be given. Then, for any $w\in\mathbb{W}$,
$\tau\ge0$, 
\begin{enumerate}
\item the map
\[
\mathcal{H}_{H}\ni h\mapsto(\partial_{1,h}F)(w,\tau)\in \mathbb{R}
\]
is a bounded linear functional on $\mathcal{H}_{H}$,
\item there exists a unique element $\nabla_1 F(w,\tau)\in\mathcal{H}_{H}$
given, for any $t\ge0$, by 
\begin{equation}\label{eq:gradient-repretion-ggBm}
\big(\nabla_1 F(w,\tau)\big)(t)=\sum_{i=1}^{n}\partial_{i}f\big(w(t_{1}\tau_H),\dots,w(t_{n}\tau_H)\big)R_H(t_{i}\tau_H,t),
\end{equation}
and 
\begin{equation}\label{eq:gradient-repretion-dot-ggBm}
\big(\dot{\nabla}_1 F(w,\tau)\big)(t)=\sum_{i=1}^{n}\partial_{i}f\big(w(t_{1}\tau_H),\dots,w(t_{n}\tau_H)\big)K_H(t_{i}\tau_H,t)
\end{equation}
satisfying 
\begin{equation}\label{eq:derivative-repretion-ggBm}
(\partial_{1,h}F)(w,\tau)=\big(\nabla_1 F(w,\tau),h\big)_{\mathcal{H}_{H}},\quad w\in\mathbb{W},\;\tau\ge0.
\end{equation}
\end{enumerate}
\end{prop}

\begin{proof}
1. Let $F\in\mathcal{F}C_{b,H}^\infty$ be given. The linearity of
the map is obvious. For every $h\in\mathcal{H}_{H}$, using the
representation (\ref{eq:h-rep-fBm}) we have
\begin{eqnarray*}
\big|(\partial_{1,h}F)(w,\tau)\big| & = & \left|\sum_{i=1}^{n}\partial_{i}f\big(w(t_{1}\tau_H),\dots,w(t_{n}\tau_H)\big)h(t_{i}\tau_H)\right|\\
 & \le & \sum_{i=1}^{n}\|\partial_{i}f\|_{\infty}\left|\int_{0}^{t_{i}\tau_H}K_H(t_i\tau_H,s)\dot{h}(s)\,\mathrm{d}s\right|.
\end{eqnarray*}
An application of the Cauchy-Schwarz inequality, the fact that
$\|\dot{h}\|_{2}=\|h\|_{\mathcal{H}_{H}}$, and the equality \eqref{eq:cov-integral-kernel} implies
\begin{equation}
\big|(\partial_{1,h}F)(w,\tau)\big| \leq \sum_{i=1}^{n}\|\partial_{i}f\|_{\infty}t_{i}^H\tau^{1/2}
\|h\|_{\mathcal{H}_H}.\label{eq:est-CM-norm-ggBm}
\end{equation}
2. The equality (\ref{eq:derivative-repretion-ggBm}) follows from 1.~and
the Riesz representation theorem. From (\ref{eq:derivative-cylinder-fction-ggBm})
and (\ref{eq:h-rep-fBm}) we obtain 
\begin{eqnarray*}
(\partial_{1,h}F)(w,\tau) & = & \sum_{i=1}^{n}\partial_{i}f\big(w(t_{1}\tau_H),\dots,w(t_{n}\tau_H)\big)h(t_{i}\tau_H)\\
 & = & \sum_{i=1}^{n}\partial_{i}f\big(w(t_{1}\tau_H),\dots,w(t_{n}\tau_H)\big)\int_{0}^{t_{i}\tau_H}K_H(t_i\tau_H,s)\dot{h}(s)\,\mathrm{d}s\\
 & = & \int_{0}^{\infty}\sum_{i=1}^{n}\partial_{i}f\big(w(t_{1}\tau_H),\dots,w(t_{n}\tau_H)\big)K_H(t_i\tau_H,s)\dot{h}(s)\,\mathrm{d}s.
\end{eqnarray*}
Therefore, the equality (\ref{eq:gradient-repretion-dot-ggBm})
is immediate and (\ref{eq:gradient-repretion-ggBm}) follows from \eqref{eq:CalK-K-equal-R}.
\end{proof}

The estimate \eqref{eq:est-CM-norm-ggBm} is used to demonstrate that the directional derivative $\partial_{1,h}F$, $F\in\mathcal{F}C_{b,H}^{\infty}$, $h\in\mathcal{H}_{H}$, is an element of $L^{p}(\mathbb{W}\times\mathbb{R}^+,\mathcal{F}^{X_{\beta,2H,T}},\mathbb{P}_{H}\otimes\mathbb{P}_{Y_\beta})$ 
for any $p\ge1$, where $\mathcal{F}^{X_{\beta,2H,T}}:=\sigma\big(X_{\beta,2H,T}(t),t\in\left[0,T\right]\big)$.
\begin{prop}\label{prop:direct-deriv-Lp-ggBm} Let $F\in\mathcal{F}C_{b,H}^{\infty}$
and $p\ge1$ be given. Then
\begin{enumerate}
\item $\partial_{1,h}F\in L^{p}(\mathbb{W}\times\mathbb{R}^+,\mathcal{F}^{X_{\beta,2H,T}},\mathbb{P}_{H}\otimes\mathbb{P}_{Y_\beta})$, for any $h\in\mathcal{H}_{H}$,

\item $\nabla_1 F\in L_{\mathcal{H}_{H}}^{p}(\mathcal{F}^{X_{\beta,2H,T}},\mathbb{P}_{H}\otimes\mathbb{P}_{Y_\beta}):=L^{p}(\mathbb{W}\times\mathbb{R}^+,\mathcal{F}^{X_{\beta,2H,T}},\mathbb{P}_{H}\otimes\mathbb{P}_{Y_\beta};\mathcal{H}_{H})$. 
\end{enumerate}
\end{prop}

\begin{proof}
1. Let $F\in\mathcal{F}C_{b,H}^{\infty}$ be given. It follows from \eqref{eq:cylinder-ggBm} and \eqref{eq:derivative-cylinder-fction-ggBm} that $F$ and $\partial_{1,h}F$ are $\mathcal{F}^{X_{\beta,2H,T}}$-measurable. Using the estimate~(\ref{eq:est-CM-norm-ggBm}) we obtain 
\begin{eqnarray*}
\|\partial_{1,h}F\|_{p}^{p}  & \le & \int_{\mathbb{W}\times\mathbb{R}^{+}}\left|\sum_{i=1}^{n}\|\partial_{i}f\|_{\infty}t_{i}^H\tau^{1/2}\|h\|_{\mathcal{H}_{H}}\right|^{p}\mathrm{d}(\mathbb{P}_{H}\otimes\mathbb{P}_{Y_{\beta}})(w,\tau)\\
 & \le & n^{p-1}\|h\|_{\mathcal{H}_{H}}^{p}\sum_{i=1}^{n}\|\partial_{i}f\|_{\infty}^{p}t_{i}^{pH}\int_{0}^{\infty}\tau^{p/2}\,\mathrm{d}\mathbb{P}_{Y_{\beta}}(\tau).\\
 & = & n^{p-1}\|h\|_{\mathcal{H}_{H}}^{p}\sum_{i=1}^{n}\|\partial_{i}f\|_{\infty}^{p}t_{i}^{pH}\frac{\Gamma(p/2+1)}{\Gamma(\beta p/2+1)}<\infty.
\end{eqnarray*}
The second estimate is the convexity inequality, and the last equality
follows from (\ref{eq:Mbeta-moments}).

\noindent 2. The $\mathcal{F}^{X_{\beta,2H,T}}$-measurability of $\nabla_1 F$  and $\dot{\nabla}_1 F$ can be easily derived from \eqref{eq:gradient-repretion-ggBm} and \eqref{eq:gradient-repretion-dot-ggBm}.
From (\ref{eq:gradient-repretion-dot-ggBm}) and again the convexity inequality we infer 
\begin{eqnarray*}
\|\nabla_1 F(w,\tau)\|_{\mathcal{H}_{H}}^{2} & = & \int_{0}^{\infty}\big|\big(\dot{\nabla}_1 F(w,\tau)\big)(t)\big|^{2}\,\mathrm{d}t\\
 & = & \int_{0}^{\infty}\left|\sum_{i=1}^{n}\partial_{i}f\big(w(t_{1}\tau_H),\dots,w(t_{n}\tau_H)\big)K_H(t_i\tau_H,t)\right|^{2}\mathrm{d}t\\
 & \le & n\sum_{i=1}^{n}t_i^{2H}\tau\|\partial_{i}f\|_\infty^2.
\end{eqnarray*}
This implies 
\begin{eqnarray*}
\|\nabla_1 F(w,\tau)\|_{\mathcal{H}_{H}}^{p} & \le & \left(n\sum_{i=1}^{n}t_{i}^{2H}\tau\|\partial_{i}f\|_\infty^2\right)^{p/2}\\
 & \le & n^{p/2\vee(p-1)}\sum_{i=1}^{n}t_{i}^{pH}\tau^{p/2}\|\partial_{i}f\|_\infty^{p}.
\end{eqnarray*}
Finally, the norm of $\|\nabla_1 F\|_{L_{\mathcal{H}_{H}}^{p}(\mathcal{F}^{X_{\beta,2H,T}},\mathbb{P}_{H}\otimes\mathbb{P}_{Y_\beta})}$
is computed as

\begin{eqnarray*}
\|\nabla_1 F\|_{L_{\mathcal{H}_{H}}^{p}(\mathcal{F}^{X_{\beta,2H,T}},\mathbb{P}_{H}\otimes\mathbb{P}_{Y_\beta})}^{p} 
 & = & \int_{W\times\mathbb{R}^{+}}\|(\nabla_1 F)(w,\tau)\|_{\mathcal{H}_{H}}^{p}\,\mathrm{d}(\mathbb{P}_{H}\otimes \mathbb{P}_{Y_{\beta}})(w,\tau)\\
 & \le & n^{p/2\vee(p-1)}\int_{W\times\mathbb{R}^{+}}\sum_{i=1}^{n}t_{i}^{pH}\tau^{p/2}\|\partial_{i}f\|_\infty^2\mathrm{d}(\mathbb{P}_{H}\otimes \mathbb{P}_{Y_{\beta}})(w,\tau)\\
 & \le & n^{p/2\vee(p-1)}\sum_{i=1}^{n}\|\partial_{i}f\|_{\infty}^{p}t_{i}^{pH}\int_{\mathbb{R}^{+}}\tau^{p/2}\mathrm{d}\mathbb{P}_{Y_{\beta}}(\tau)\\
 & = & n^{p/2\vee(p-1)}\sum_{i=1}^{n}\|\partial_{i}f\|_{\infty}^{p}t_{i}^{pH}\frac{\Gamma(p/2+1)}{\Gamma(\beta p/2+1)}<\infty.
\end{eqnarray*}
This completes the proof.
\end{proof}
We now get the integration by parts formula for ggBm that will be used to prove the closability of the derivative operators $\partial_{1,h}$ and $\nabla_1$ on $L^p(\mathbb{P}_H\otimes\mathbb{P}_{Y_\beta})$.

\begin{thm}\label{thm:IbP-ggBm}
    Let $T>0$ be given. For any $h\in\mathcal{H}_{H}$ and $F,G\in\mathcal{F}C_{b,H}^{\infty}$, we have
    \begin{equation}\label{eq:IbP-ggBm}  \mathbb{E}\big[G\partial_{1,h}F\big]=\mathbb{E}\big[F\partial_{1,h}^*G\big],
    \end{equation}
    where
    \[
    \partial_{1,h}^*G=-\partial_{1,h}G+G \left( \int_0^{\cdot} \dot{h}(t)\,\mathrm{d}\widetilde{W}(t) \right)(T\mathcal{Y}_\beta^{1/(2H)}).
    \]
\end{thm}

\begin{proof} 
The proof is analogous to the one of Theorem \ref{thm:IbP-gBm} with slight modifications.
\end{proof}

From the identity \eqref{eq:IbP-ggBm} we obtain the following
\begin{thm}
\label{thm:closable-direct-derivative-ggBm}For every $p,q\geq 1$ and $h\in\mathcal{H}_{H}$
the derivative operator $\partial_{1,h}:\mathcal{F}C_{b,H}^{\infty}\longrightarrow L^{p}(\mathbb{W}\times\mathbb{R}^{+},\mathcal{F}^{X_{\beta,2H,T}},\mathbb{P}_{H}\otimes\mathbb{P}_{Y_\beta})$
is closable in $L^{q}(\mathbb{W}\times\mathbb{R}^{+},\mathcal{F}^{X_{\beta,2H,T}},\mathbb{P}_{H}\otimes\mathbb{P}_{Y_\beta})$.  
\end{thm}

\begin{proof}
Let $p,q\ge1$ be given. We have to show that if $(F_{n})_{n\in\mathbb{N}}\subset\mathcal{F}C_{b,H}^{\infty}$
is a sequence such that $F_{n}\longrightarrow0$ in $L^{q}(\mathcal{F}^{X_{\beta,2H,T}},\mathbb{P}_{H}\otimes\mathbb{P}_{Y_\beta})$
and $\partial_{1,h}F_{n}\longrightarrow Z$ in $L^{p}(\mathcal{F}^{X_{\beta,2H,T}},\mathbb{P}_{H}\otimes\mathbb{P}_{Y_\beta})$, then
$Z=0$ in $L^{p}(\mathcal{F}^{X_{\beta,2H,T}},\mathbb{P}_{1/2}\otimes\mathbb{P}_{Y_\beta})$.

Let $G\in\mathcal{F}C_{b,H}^{\infty}$
be fixed and $r\ge1$. According to Proposition~\ref{prop:direct-deriv-Lp-ggBm}-1
we have $\partial_{1,h}G\in L^{r}(\mathcal{F}^{X_{\beta,2H,T}},\mathbb{P}_{H}\otimes\mathbb{P}_{Y_\beta})$. In addition, using the Burkholder--Davis--Gundy
inequality, there exists a constant $C_r$ dependent  only in $r$ such that  
\begin{eqnarray*}
 & & \int_{0}^{\infty}\int_{\mathbb{W}}\left|\left(\int_{0}^{T\tau_H}\dot{h}(t)\,\mathrm{d}\widetilde{W}(t)\right)(w)\right|^{r}\mathrm{d}\mathbb{P}_{H}(w)\,\mathrm{d}\mathbb{P}_{Y_{\beta}}(\tau)\\
 & \le & C_{r}\int_{0}^{\infty}\left(\int_{0}^{T\tau_H}|\dot{h}(t)|^{2}\,\mathrm{d}t\right)^{r/2}\mathrm{d}\mathbb{P}_{Y_{\beta}}(\tau)\\
 & \le & C_{r}\|\dot{h}\|_{2}^{r/2}=C_{r}\|h\|_{\mathcal{H}_{H}}^{r/2}.
\end{eqnarray*}
We know that the Wiener integral $\int_{0}^{\cdot}\dot{h}(t)\,\mathrm{d}\widetilde{W}(t)$ is the limit in probability of Riemann sums; see \cite{Revuz-Yor-94}. Furthermore, the $\sigma$ -algebras of $W^H$ and $\widetilde{W}$ are the same; see Corollary~4.1 in \cite{Decreusefond22}. Consequently, $\left(\int_{0}^{\cdot}\dot{h}(t)\,\mathrm{d}\widetilde{W}(t)\right)(T\mathcal{Y}_{\beta})$ is $\mathcal{F}^{X_{\beta,2H,T}}$-measurable.
Thus, 
\[
\partial_{1,h}^{*}G=-\partial_{1,h}G+G\left(\int_{0}^{\cdot}\dot{h}(t)\,\mathrm{d}\widetilde{W}(t)\right)(T\mathcal{Y}_{\beta})\in L^{r}(\mathcal{F}^{X_{\beta,2H,T}},\mathbb{P}_{H}\otimes\mathbb{P}_{Y_\beta}).
\]
Therefore, the integration by parts (\ref{eq:IbP-ggBm}) and the fact that $\partial_{1,h}^*G\in L^{q'}(\mathbb{P}_{H}\otimes\mathbb{P}_{Y_\beta})$, $q'$ is the conjugate exponent of $q$, yields 
\[
\mathbb{E}[GZ]=\lim_{n\to\infty}\mathbb{E}[G\partial_{1,h}F_{n}]=\lim_{n\to\infty}\mathbb{E}[F_{n}\partial_{1,h}^{*}G]=0.
\]
We can deduce that $Z=0$ from the facts that the $\sigma$-algebra $\mathcal{F}^{X_{\beta,2H,T}}$ is generated by the elements of $\mathcal{F}C_{b,H}^{\infty}$ and the density of  $\mathcal{F}C_{b,H}^{\infty}$ in $L^{p}(\mathcal{F}^{X_{\beta,2H,T}},\mathbb{P}_{H}\otimes\mathbb{P}_{Y_\beta})$.
\end{proof}

Finally, we are ready to prove the closability of the operator $\nabla_1$ on $L^p(\mathbb{P}_H\otimes\mathbb{P}_{Y_\beta})$, which is the content of the following theorem.
\begin{thm}
\label{thm:closable-gradient-ggBm}For every $p,q\geq 1$, the operator $\nabla_1:\mathcal{F}C_{b,H}^{\infty}\longrightarrow L_{\mathcal{H}_{H}}^{p}(\mathcal{F}^{X_{\beta,2H,T}},\mathbb{P}_{H}\otimes\mathbb{P}_{Y_\beta})$
is closable on $L^{q}(\mathbb{W}\times\mathbb{R}^+,\mathcal{F}^{X_{\beta,2H,T}},\mathbb{P}_{H}\otimes\mathbb{P}_{Y_\beta})$.
\end{thm}

\begin{proof}
Let $p,q\ge1$ and $(F_{n})_{n\in\mathbb{N}}\subset\mathcal{F}C_{b,H}^{\infty}$
be a sequence such that $F_{n}\longrightarrow0$ in $L^{q}(\mathcal{F}^{X_{\beta,2H,T}},\mathbb{P}_{H}\otimes\mathbb{P}_{Y_\beta})$
and $\nabla_1 F_{n}\longrightarrow Z$ in $L_{\mathcal{H}_{H}}^{p}(\mathcal{F}^{X_{\beta,2H,T}},\mathbb{P}_{H}\otimes\mathbb{P}_{Y_\beta})$.
We have to show that $Z=0$ in $L_{\mathcal{H}_{H}}^{p}(\mathcal{F}^{X_{\beta,2H,T}},\mathbb{P}_{H}\otimes\mathbb{P}_{Y_\beta})$.
First, notice that for any $h\in\mathcal{H}_{H}$ we have
\[
\partial_{1,h}F_{n}=(\nabla_1 F_{n},h)_{\mathcal{H}_{H}}\longrightarrow(Z,h)_{\mathcal{H}_{H}},\;\mathrm{in}\;L^{p}(\mathcal{F}^{X_{\beta,2H,T}},\mathbb{P}_{H}\otimes\mathbb{P}_{Y_\beta})\;\mathrm{as}\;n\to\infty.
\]
Since $\partial_{1,h}F_{n}$ is closable (cf.~Theorem~\ref{thm:closable-direct-derivative-ggBm}),
then $(Z,h)_{\mathcal{H}_{H}}=0$ in $L^{p}(\mathcal{F}^{X_{\beta,2H,T}},\mathbb{P}_{H}\otimes\mathbb{P}_{Y_\beta})$. 
Hence,
\[
\mathbb{E}[G(Z,h)_{\mathcal{H}_{H}}]=0,\quad\forall G\in\mathcal{F}C_{b,H}^{\infty}.
\]
By linearity we obtain 
\[
\mathbb{E}\left[\left(Z,\sum_{i=1}^n G_ih_i \right)_{\mathcal{H}_{H}}\right]=0,
\]
for any element of the set 
\[
S(\mathcal{H}_{H}):=\left\{ \sum_{i=1}^{n}G_{i}h_{i}\,\middle|\, n\in\mathbb{N},\, G_{i}\in\mathcal{F}C_{b,H}^{\infty}\right\}, 
\]
where $\{h_{i}\in\mathcal{H}_{H},\;i\in\mathbb{N}\}$ is an orthonormal
basis of $\mathcal{H}_{H}$. Since $S(\mathcal{H}_{H})$ is dense in $L_{\mathcal{H}_{H}}^{p}(\mathcal{F}^{X_{\beta,2H,T}},\mathbb{P}_{H}\otimes\mathbb{P}_{Y_\beta})$, 
we conclude that $Z=0$ in $L_{\mathcal{H}_{H}}^{p}(\mathcal{F}^{X_{\beta,2H,T}},\mathbb{P}_{1/2}\otimes\mathbb{P}_{Y_\beta})$.
\end{proof}

\section{Appendix}
In this section, we establish the existence of the integrals that appear in the proof of Theorems~\ref{thm:IbP-gBm} and \ref{thm:IbP-ggBm}. 

\begin{lem}\label{lem:h-integral-PYbeta}
Let $H\in [1/2,1)$ and $h\in \mathcal{H}_{H}$ be given. Then, for every $t\geq 0$, we have
\begin{enumerate}
    \item The random variable $(w,\tau)\mapsto\displaystyle\left(\int_{0}^{t\tau}\dot{h}(s)\,\mathrm{d}W^{1/2}(s)\right)(w)$ belongs to $L^{2}\big(\mathbb{W}\times\mathbb{R}^+,\mathcal{B}(\mathbb{W})\otimes\mathcal{B}(\mathbb{R}^+),\mathbb{P}_{1/2}\otimes\mathbb{P}_{Y_{\beta}}\big)$, $H=1/2$.
    \item The random variable $(w,\tau)\mapsto \displaystyle\left(\int_{0}^{t\tau}\dot{h}(s)\,\mathrm{d}\widetilde{W}(s)\right)(w)$ belongs to $ L^{2}\big(\mathbb{W}\times\mathbb{R}^+,\mathcal{B}(\mathbb{W})\otimes\mathcal{B}(\mathbb{R}^+),\mathbb{P}_{H}\otimes\mathbb{P}_{Y_{\beta}}\big)$, $H\in(1/2,1)$.
\end{enumerate}
\end{lem}
\begin{proof}
It is a consequence of the Tonelli theorem and It\^o's isometry formula. Indeed, since $W^{1/2}$ and $\widetilde{W}$ are Brownian motions under $\mathbb{P}_{1/2}$ and $\mathbb{P}_{H}$, $H\in (1/2,1)$, respectively, we obtain the bound

\begin{eqnarray*}
\left\|\int_{0}^{t\cdot}\dot{h}(s)\,\mathrm{d}\widetilde{W}(s)\right\|^2_{L^{2}(\mathbb{P}_{H}\otimes\mathbb{P}_{Y_{\beta}})}
&=&\left\|\int_{0}^{t\cdot}\dot{h}(s)\,\mathrm{d}W^{1/2}(s)\right\|^2_{L^{2}(\mathbb{P}_{1/2}\otimes\mathbb{P}_{Y_{\beta}})}\\&=&\int_{\mathbb{W}\times\mathbb{R}^{+}}\left(\int_{0}^{t\tau}\dot{h}(s)\,\mathrm{d}W^{1/2}(s)\right)^{2}\!\! (w)\,\mathbb{P}_{1/2}(\mathrm{d}w)\,\mathbb{P}_{Y_{\beta}}(\mathrm{d}\tau) \\ &=&\int_{\mathbb{R}^{+}}\left(\int_{0}^{t\tau}\dot{h}^{2}(s)\,\mathrm{d}s\right)\mathrm{d}\mathbb{P}_{Y_{\beta}}(\tau)\\
&\leq& \|\dot{h}\|_2.
\end{eqnarray*}
\end{proof}

\subsection*{Acknowledgments}

We express our gratitude to Hassan El-Essaky for the
hospitality during an enjoyable stay at the Polydisciplinary Faculty
of Safi on the occasion of the 3rd International Conference on Applied Mathematics and Computer Science during which part of this work was done. The last author expresses gratitude to Professor Michael R{\"o}ckner for the invitation to the University of Bielefeld when this research was concluded and for the opportunity to give a talk at the Bielefeld Stochastic Afternoon seminar. The second author is grateful to CIMA--University of Madeira for providing a pleasant stay during the summer of 2023, during which a portion of this work was completed.

\subsection*{Founding}
The Center for Research in Mathematics and Applications (CIMA) has partially funded this work related to the Statistics, Stochastic Processes
and Applications (SSPA) group, through the grant UIDB/MAT/04674/2020,  \url{https://doi.org/10.54499/UIDB/04674/2020} 
of FCT-Funda{\c c\~a}o para a Ci{\^e}ncia e a Tecnologia, Portugal. Funded by the Deutsche Forschungsgemeinschaft (DFG, German Research Foundation) -- Project-ID 317210226 -- SFB 1283.

\end{document}